\documentclass[11pt,oneside, a4paper]{amsart} 

\usepackage[english]{babel}
\usepackage[T1]{fontenc}
\usepackage{amsmath}
\usepackage{amssymb}
\usepackage{float}
\usepackage{amsfonts}
\usepackage{mathtools}
\usepackage[utf8]{inputenc}
\usepackage[mathcal]{eucal}
\usepackage{enumerate}
\usepackage{amsthm}
\usepackage{hyperref}
\usepackage[totalwidth=14cm,totalheight=20cm]{geometry}
\usepackage{epsfig,fancyhdr,color}
\usepackage{multicol} 

\newtheorem{cor}{Corollary}[section]

\newtheorem{teo}[cor]{Theorem}

\newtheorem{prop}[cor]{Proposition}

\newtheorem{lemma}[cor]{Lemma}

\theoremstyle{definition}

\theoremstyle{remark}
\newtheorem{remark}[cor]{Remark}
\newtheorem*{remark*}{Remark}

\newcommand{\Pp}{\mathbb{P}}
\newcommand{\R}{\mathbb{R}}

\newcommand{\C}{\mathbb{C}}

\newcommand{\h}{\mathbb{H}}

\newcommand{\diag}{\mathrm{diag}}

\newcommand{\SL}{\mathrm{SL}}

\newcommand{\dPSL}{\mathbb{P}SL(2,\mathbb{R})\times \mathbb{P}SL(2,\mathbb{R})}
\newcommand{\PSL}{\mathbb{P}SL}

\newcommand{\Ima}{\mathrm{Im}}

\newcommand{\Isom}{\mathrm{Isom}}

\newcommand{\SO}{\mathrm{SO}}

\newcommand{\Id}{\mathrm{Id}}

\newcommand{\AdS}{\mathrm{AdS}}
\newcommand{\Imm}{\mathcal{I}\text{m}}
\newcommand{\Teich}{\mathpzc{T}}
\renewcommand{\Re}{\mathcal{R}\text{e}}
\newcommand{\hol}{\mathrm{hol}}

\DeclareMathAlphabet{\mathpzc}{OT1}{pzc}{m}{it}

\title[Regular GHM AdS structures]{Regular globally hyperbolic maximal \\ anti-de Sitter structures}

\author{Andrea Tamburelli}

\begin{document}

\begin{abstract}
Let $\Sigma$ be a connected, oriented surface with punctures and negative Euler characteristic. We introduce regular globally hyperbolic anti-de Sitter structures on $\Sigma \times \R$ and provide two parameterisations of their deformation space: as an enhanced product of two copies of the Fricke space of $\Sigma$ and as the bundle over the Teichm\"uller space of $\Sigma$ whose fibre consists of meromorphic quadratic differentials with poles of order at most $2$ at the punctures. 
\end{abstract}

\maketitle

\setcounter{tocdepth}{1}

\tableofcontents

\section*{Introduction}
Anti-de Sitter geometry is the Lorentzian analogue of hyperbolic geometry, being it the local model of Lorentzian manifolds of constant sectional curvature $-1$. After the pioneering work of Mess (\cite{Mess}), three-dimensional anti-de Sitter geometry has attracted the interest of the mathematical community due to its connections with Teichm\"uller theory (\cite{bsk_multiblack}, \cite{BonSchlGAFA2009}) and hyperbolic geometry (\cite{bonschlfixed}). A special class of anti-de Sitter manifolds, called globally hyperbolic maximal compact (GHMC), turned out to share many anologies with hyperbolic quasi-Fuchsian manifolds: they are topologically a product $S\times \R$, where $S$ is a closed, connected, oriented surface of genus at least two, and the deformation space $\mathcal{GH}(S)$ of such structures is parameterised by the product of two copies of the Teichm\"uller space of $S$. While the theory is well-developed and the geometry of these manifolds is well-understood when $S$ is closed (\cite{volumeAdS}, \cite{foliationCMC}, \cite{folKsurfaces}, \cite{entropy}), the analogous picture for surfaces with punctures is not well-known. \\

In this paper, we propose a definition of a class of globally hyperbolic maximal anti-de Sitter structures on the product $\Sigma \times \R$, where $\Sigma$ is a closed, connected, oriented surface with $k$ punctures and negative Euler characteristic, that we call \emph{regular}. 
Our construction is inspired by the parameterisation of $\mathcal{GH}(S)$ found by Krasnov and Schlenker (\cite{Schlenker-Krasnov}): they exploited the uniqueness of the maximal surface embedded in a GHMC anti-de Sitter manifold $M$ to find a homeomorphism between $\mathcal{GH}(S)$ and the cotangent bundle of the Teichm\"uller space of $S$. They associated to $M$ the conformal class of the induced metric and the holomorphic quadratic differential that determines the second fundamental form of the unique maximal surface embedded in $M$. Our idea consists in studying what happens when replacing holomorphic quadratic differentials with meromorphic quadratic differentials that have at most second order poles at the punctures. There are many reasons to consider second order poles singularities: they naturally appear in the compactification of the cotangent bundle of the moduli space of Riemann surfaces (\cite{Wolpert_WPcompletion}) and they are related to parabolic Higgs bundles with regular singularities as studied by Simpson (\cite{Simpson_irregular}). Moreover, in a companion paper (\cite{compactify_AdS}), we will show that regular GHM anti-de Sitter structures naturally appear as limits of GHMC manifolds along pinching sequences. 
\\

We first show the existence of a maximal surface starting from the aforementioned data:\\
\\
{\bf Theorem A.}{\it \ Given a conformal structure on $\Sigma$ and a meromorphic quadratic differential $q$ with at most second order poles at the punctures, there exists a unique (up to global isometries) conformal equivariant maximal embedding $\tilde{\sigma}:\tilde{\Sigma}\rightarrow \AdS_{3}$ into anti-de Sitter space whose second fundamental form is the real part of $q$.}\\
\\
\indent The embedding $\tilde{\sigma}$ comes together with a representation $\rho:\pi_{1}(\Sigma) \rightarrow \Isom(\AdS_{3})$ and the main bulge of the paper deals with understanding how the residue (i.e. the leading coefficient in the Laurent expansion of $q$ at the punctures) determines the representation, and with describing the maximal globally hyperbolic domain of discontinuity for the action of $\rho(\pi_{1}(\Sigma))$. Recall that, by identifying $\Isom(\AdS_{3})$ with $\dPSL$, $\rho$ is equivalent to a couple of representations $\rho_{l,r}:\pi_{1}(\Sigma) \rightarrow \PSL(2,\R)$. We prove the following:\\
\\
{\bf Theorem B.}{\it \ Let $R_{i}$ be the residues of the meromorphic quadratic differential $q$ on $\Sigma$ and let $\tilde{\sigma}$ be the conformal maximal embedding of Theorem A. Then $\rho_{l}$ and $\rho_{r}$ are holonomies of hyperbolic structures on $\Sigma$ and the behaviour along the peripheral curves $\gamma_{i}$ is determined as follows: 
\vspace{-0.6cm}
\begin{enumerate}[i)]
	\item if $\Re(R_{i})\neq 0$ and $\Imm(R_{i})\neq 0$ then both $\rho_{r}(\gamma_{i})$ and $\rho_{l}(\gamma_{i})$ are hyperbolic;
	\item if $\Re(R_{i})=0$ and $\Imm(R_{i})\neq 0$ then exactly one between $\rho_{r}(\gamma_{i})$ and $\rho_{l}(\gamma_{i})$ is parabolic and the other is hyperbolic;
	\item if $R_{i}=0$ then $\rho_{r}(\gamma_{i})$ and $\rho_{l}(\gamma_{i})$ are both parabolic.
\end{enumerate}
}

\indent Unlike the closed case, we point out that the same holonomy can be realised by different meromorphic quadratic differentials. More precisely, if we only change the sign of the real part of the residue, the representation $\rho$ does not change. What helps us distinguish these cases is the geometry of the boundary at infinity of the associated maximal surface. Recall that the boundary at infinity of anti-de Sitter space can be identified with $\R\Pp^{1} \times \R\Pp^{1}$ and the action of $\rho=(\rho_{l}, \rho_{r})$ extends naturally on each factor. \\
\\
{\bf Theorem C.}{\it \ The boundary at infinity of $\tilde{\sigma}(\tilde{\Sigma})$ is a locally achronal curve that contains the closure of the set of couples of attractive fixed points of $(\rho_{l},\rho_{r})$. This set is completed to a topological circle by inserting, in a $\rho$-equivariant way, a future-directed or a past-directed saw-tooth for each hyperbolic end depending on the sign of the real part of the corresponding residue.}\\
\\
Recall that a saw-tooth, as defined in \cite{folKsurfaces}, is a "vee" on the boundary at infinity of anti-de Sitter space formed by a segment belonging to the right-foliation $\{*\}\times \R\Pp^{1}$ concatenated with a segment belonging to the left-foliation $\R\Pp^{1}\times \{*\}$ (or viceversa).\\
\\
\indent The boundary at infinity of $\tilde{\sigma}(\tilde{\Sigma})$ determines then a domain of dependence on which $\rho(\pi_{1}(\Sigma))$ acts properly discontinuously and the quotient gives the desired regular globally hyperbolic anti-de Sitter manifold diffeomorohic to $\Sigma \times \R$. Notice that, since the holonomy representation does not determine alone the structure, new data must be introduced in order to obtain the analogue of Mess' parameterisation. \\
\\
{\bf Theorem D.}{\it \ The deformation space of regular globally hyperbolic anti-de Sitter structures on $\Sigma\times \R$ is parameterised by  $(k+2)$-uples $(h_{l},h_{r}, \epsilon_{1}, \dots \epsilon_{k})$ where $h_{l,r}$ are hyperbolic metrics on $\Sigma$ such that each end corresponds to a cusp or a geodesic boundary, and $\epsilon_{j}$ is a decoration on each puncture so that
\[
	\epsilon_{j}=\begin{cases}
		\pm 1 \ \ \ \text{if the puncture $p_{j}$ is a geodesic boundary for both $h_{r}$ and $h_{l}$} \\
	0 \ \ \ \ \text{otherwise}
		\end{cases}
\]}

\indent As an applicaton of this theory, we describe a class of minimal Lagrangian maps between hyperbolic surfaces with cusps and geodesic boundary. An orientation preseving diffeomorphism $m:(\Sigma, h) \rightarrow (\Sigma,h')$ between hyperbolic surfaces is minimal Lagrangian if there exists a Riemann surface $X$ and harmonic maps $f:X \rightarrow (S,h)$ and $f':X \rightarrow (S,h')$ with opposite Hopf differentials such that $m=f'\circ f^{-1}$. These are in one-to-one correspondence with $(\rho_{l}, \rho_{r})$-equivariant maximal surfaces in anti-de Sitter space via the Gauss map (\cite{Schlenker-Krasnov}, \cite{bon_schl}): the Riemann surface $X$ is determined by the conformal structure of the maximal surface $S$, $h$ and $h'$ are the hyperbolic metrics on $\Sigma$ with holonomy $\rho_{l}$ and $\rho_{r}$, respectively and the harmonic maps $f$ and $f'$ are the projections of the Gauss map $G:\tilde{S}\rightarrow \h^{2}\times \h^{2}$ onto the left and right factor. We thus deduce the following:\\
\\ 
{\bf Theorem E.}{\it \ Let $(\Sigma, h)$ (resp. $(\Sigma,h')$) be hyperbolic surfaces with $b$ (resp. $b'$) geodesic boundaries and $p$ (resp. $p'$) cusps. Then there exist $2^{\min(b,b')}$ minimal Lagrangian diffeomorphisms from $(\Sigma, h)$ to $(\Sigma, h')$ that do not extend to the boundaries.}

\subsection*{Outline of the paper}In Section \ref{sec:background} we recall well-known facts about anti-de Sitter geometry and meromorphic quadratic differentials on surfaces. In Section \ref{sec:existence_max} we prove the existence of an equivariant maximal embedding starting from the data of a conformal structure on $\Sigma$ and a meromorphic quadratic differential with poles of order at most two at the punctures. The associated holonomy representation is described in Section \ref{sec:holonomy}. Theorem C is proved in Section \ref{sec:boundary}. We then parameterise the deformation space of regular GHM anti-de Sitter structures in Section \ref{sec:AdS_regularGHM}. The connection with minimal Lagrangian maps is explained in Section \ref{sec:application}.

\section{Background material}\label{sec:background} We recall here some well-known facts about anti-de Sitter geometry and (meromorphic) quadratic differentials on Riemann surfaces that will be used in the sequel. Throughout the paper, we will denote with $\overline{\Sigma}$ a closed, connected, oriented surface and with $\Sigma=\overline{\Sigma}\setminus \{p_{1}, \dots, p_{k}\}$ a surface with a finite number of punctures. We will always assume that $\chi(\Sigma)<0$. Moreover, we will denote with $\mathcal{T}(\Sigma)$ the Teichm\"uller space of $\Sigma$, i.e. the space of marked complete hyperbolic structures of finite area on $\Sigma$ up to isotopy. 

\subsection{Anti-de Sitter geometry} Consider the vector space $\R^{4}$ endowed with a bilinear form of signature $(2,2)$
\[
	\langle x,y\rangle= x_{0}y_{0}+x_{1}y_{1}-x_{2}y_{2}-x_{3}y_{3} \ .
\]
We denote 
\[
	\widehat{\AdS}_{3}=\{ x \in \R^{4} \ | \ \langle x , x \rangle=-1 \} \ .
\]
It can be easily verified that $\widehat{\AdS}_{3}$ is diffeomorphic to a solid torus and the restriction of the bilinear form to the tangent space at each point endows $\widehat{\AdS}_{3}$ with a Lorentzian metric of constant sectional curvature $-1$. Anti-de Sitter space is then defined as
\[
	\AdS_{3}=\Pp(\{x \in \R^{4} \ | \ \langle x,x\rangle < 0\})\subset \R\Pp^{3} \ .
\]
The natural map $\pi:\widehat{\AdS}_{3} \rightarrow \AdS_{3}$ is a two-sheeted covering and we endow $\AdS_{3}$ with the induced Lorentzian structure. The isometry group of $\widehat{\AdS_{3}}$ that preserves the orientation and the time-orientation is $\SO_{0}(2,2)$, the connected component of the identity of the group of linear transformations that preserve the bilinear form of signature $(2,2)$. \\

The boundary at infinity of anti-de Sitter space is naturally identified with 
\[
	\partial_{\infty}\AdS_{3}=\Pp(\{ x \in \R^{4} \ | \ \langle x,x\rangle=0\}) \ .
\]
It coincides with the image of the Segre embedding $s:\R\Pp^{1}\times \R\Pp^{1} \rightarrow \R\Pp^{3}$, and thus, it is foliated by two families of projective lines, which we distinguish by calling $s(\R\Pp^{1} \times \{*\})$ the right-foliation and $s(\{*\}\times \R\Pp^{1})$ the left-foliation. The action of an isometry extends continuously to the boundary, and preserves the two foliations. Moreover, it acts on each line by a projective transformation, thus giving an identification between $\Pp\SO_{0}(2,2)$ and $\dPSL$. \\

The Lorentzian metric on $\AdS_{3}$ induces on $\partial_{\infty}\AdS_{3}$ a conformally flat Lorentzian structure. To see this, notice that the map 
\begin{align*}
	F:D \times S^{1} &\rightarrow \widehat{\AdS}_{3} \\
		(z,w) &\mapsto \left( \frac{2}{1-\|z\|^{2}}z, \frac{1+\|z\|^{2}}{1-\|z\|^{2}}w\right)
\end{align*}
is a diffeomorphism, hence $D\times S^{1}$ is a model for anti-de Sitter space if endowed with the pull-back metric
\[
	F^{*}g_{\AdS_{3}}=\frac{4}{(1-\|z\|^{2})^{2}}|dz|^{2}-\left(\frac{1+\|z\|^{2}}{1-\|z\|^{2}}\right)d\theta'^{2} \ .
\]
Therefore, by composing with the projection $\pi:\widehat{\AdS}_{3}\rightarrow \AdS_{3}$, we deduce that $\pi \circ F$ continuously extend to a homeomorphism
\begin{align*}
	\partial_{\infty}F:S^{1}\times S^{1} &\rightarrow \partial_{\infty}\AdS_{3} \\
				(z,w) &\mapsto (z,w)
\end{align*}
and in these coordinates the conformally flat Lorentzian structure is induced by the conformal class $c=[d\theta^{2}-d\theta'^{2}]$. Notice, in particular, that the light-cone at each point $p \in \partial_{\infty}\AdS_{3}$ is generated by the two lines in the left- and right- foliation passing through $p$. 

\subsection{Complete maximal surfaces in $\AdS_{3}$} \label{sec:maxAdS} Let $U\subset \h^{2}$ be a simply connected domain. We say that $\sigma:U \rightarrow \AdS_{3}$ is a space-like embedding if $\sigma$ is an embedding and the induced metric $I=\sigma^{*}g_{AdS}$ is Riemannian. The Fundamental Theorem of surfaces embedded in anti-de Sitter space ensures that such a space-like embedding is uniquely determined, up to post-composition by a global isometry of $\AdS_{3}$, by its induced metric $I$ and its shape operator $B:\sigma_{*}TU \rightarrow \sigma_{*}TU$, which satisfy
\[
	\begin{cases} 
		d^{\nabla}B=0 \ \ \ \ \ \ \ \ \ \ \  \ \ \ \ \ \ \ \ \ \ \ \ \ \ \ \text{(Codazzi equation)} \\
		K_{I}=-1-\det(B) \ \ \ \ \ \ \ \ \ \ \ \ \ \text{(Gauss equation)}
	\end{cases}
\]
where $\nabla$ is the Levi-Civita connection and $K_{I}$ is the curvature of the induced metric on $\sigma(U)$. \\

We say that $\sigma$ is a maximal embedding if $B$ is traceless. In this case, the Codazzi equation implies that the second fundamental form $II=I(B\cdot, \cdot)$ is the real part of a quadratic differential $q$, which is holomorphic for the complex structure compatible with the induced metric $I$, in the following sense. For every couple of vector fields $X$ and $Y$ on $\sigma(U)$, we have
\[
	\Re(q)(X,Y)=I(BX,Y) \ .
\]
In a local conformal coordinate $z$, we can write $q=f(z)dz^{2}$ with $f$ holomorphic and $I=e^{2u}|dz|^{2}$. Thus, $\Re(q)$ is the bilinear form that in the frame $\{\partial_{x}, \partial_{y}\}$ is represented by 
\[
	\Re(q)=\begin{pmatrix}
			\Re(f) & -\Imm(f) \\
			-\Imm(f) & -\Re(f)
		\end{pmatrix} \ ,
\]
and the shape operator can be recovered as $B=I^{-1}\Re(q)$.\\

If the induced metric is complete, the space-like condition implies that, identifying $\widehat{AdS}_{3}$ with $D\times S^{1}$ via $F$, the surface is the graph of a $2$-Lipschitz map (\cite[Proposition 3.1]{Tambu_poly}) and its boundary at infinity $\Gamma$ is a topological circle in $\partial_{\infty}\AdS_{3}$ (\cite[Corollary 3.3]{Tambu_poly}). We also have control on the causal geometry of the curve at infinity:

\begin{lemma}\label{lm:boundary} The boundary at infinity $\Gamma$ of a complete space-like surface in $\AdS_{3}$ is locally achronal. Moreover, if two points are causally related, then a light-like segment joining them is entirely contained in $\Gamma$. 
\end{lemma}
\begin{proof} Using the model of anti-de Sitter space as product $D^{2}\times S^{1}$, we know that $\Gamma$ is the graph of a $1$-Lipshitz map $f:S^{1}\rightarrow S^{1}$. Therefore, for every $\theta_{1}, \theta_{2} \in S^{1}$ we have
\[
	d_{S^{1}}(f(\theta_{1}), f(\theta_{2})) \leq d_{S^{1}}(\theta_{1}, \theta_{2})
\]
with equality if and only if $f$ is a unit-speed parameterisation of the arc between $f(\theta_{1})$ and $f(\theta_{2})$. This already shows that $\Gamma$ is locally achronal. Suppose that $p,q \in \Gamma$ are causally related, then $q$ lies in the light-cone of $p=(\theta_{0}, \theta_{0}')$ 
\[
	L(p)=\{(\theta, \theta')\in S^{1}\times S^{1} \ | \ d_{S^{1}}(\theta_{0}, \theta)=d_{S^{1}}(\theta_{0}', \theta')\}
\]
and by the previous remark $f$ must be the unit speed parameterisation of the arc between $\theta_{0}'$ and $\theta'$. It is then straightforward to check that unit speed parameterisations of arcs are light-like segments in the boundary at infinity of $\AdS_{3}$.  
\end{proof}

\subsection{GHMC anti-de Sitter manifolds} This paper deals with the moduli space of a special class of manifolds locally isometric to $\AdS_{3}$. \\

We say that an anti-de Sitter three-manifold $M$ is Globally Hyperbolic Maximal (GHM) if it contains an embedded, oriented, space-like surface $S$ that intersects every inextensible non-space-like curve in exactly one point, and if $M$ is maximal by isometric embeddings. It turns out that $M$ is necessarily diffeomorphic to a product $S\times \R$ (\cite{MR0270697}). Moreover, we say that $M$ is Cauchy Compact (C) if $S$ is closed of genus at least $2$. We denote with $\mathcal{GH}(S)$ the deformation space of GHMC anti-de Sitter structures on $S\times \R$. \\

The theory is well-developed when $S$ is closed of genus at least $2$:
\begin{teo}[\cite{Mess}] $\mathcal{GH}(S)$ is parameterised by $\Teich(S)\times \Teich(S)$.
\end{teo}

The homeomorphism is constructed as follows. Given a GHMC anti-de Sitter structure, its holonomy representation $\rho:\pi_{1}(S) \rightarrow \Isom(\AdS_{3})\cong \dPSL$ induces a couple of representations $(\rho_{l}, \rho_{r})$ by projecting onto each factor. Mess proved that both are faithful and descrete and thus define two points in $\Teich(S)$. On the other hand, given a couple of Fuchsian representations $(\rho_{l},\rho_{r})$, there exists a unique homeomorphism $\phi: \R\Pp^{1}\rightarrow \R\Pp^{1}$ such that $\rho_{r}(\gamma)\circ \phi=\phi\circ \rho_{l}(\gamma)$ for every $\gamma \in \pi_{1}(S)$. The graph of $\phi$ defines a curve $\Lambda_{\rho}$ on the boundary at infinity of $\AdS_{3}$ and Mess constructed a maximal domain of discontinuity $\mathcal{D}(\Lambda_{\rho})$ for the action of $\rho(\pi_{1}(S))$, called domain of dependence, by considering the set of points whose (projective) dual space-like plane is disjoint from $\Lambda_{\rho}$. The quotient
\[
	M=\mathcal{D}(\Lambda_{\rho})/\rho(\pi_{1}(S))
\]
is the desired GHMC anti-de Sitter manifold. 

\begin{remark}Notice that $\rho_{r,l}$ being holonomies of complete hyperbolic structures is necessary for the uniqueness of the homeomorphism $\phi$. In Section \ref{sec:AdS_regularGHM}, we will define \emph{regular} GHM anti-de Sitter structures on $\Sigma \times \R$ and provide a similar parameterisation of their deformation space.
\end{remark}

Mess introduces also the notion of convex core. This is the smallest convex subset of a GHMC anti-de Sitter manifold $M$ onto which $M$ retracts. It can be concretely realised as follows. If $\rho$ denotes the holonomy representation of $M$ and $\Lambda_{\rho}\subset \partial_{\infty}\AdS_{3}$ is the limit set of the action of $\rho(\pi_{1}(S))$, the convex core is
\[
	\mathcal{C}(M)=\mathcal{C}(\Lambda_{\rho})/\rho(\pi_{1}(S)) \ ,
\]
where $\mathcal{C}(\Lambda_{\rho})$ denotes the convex-hull of the curve $\Lambda_{\rho}$. If $M$ is Fuchsian (i.e. the left and right representations coincide), the convex core is a totally geodesic surface. Otherwise, it is a three-dimensional domain, homeomorphic to $S\times I$, the two boundary components being space-like surfaces, endowed with a hyperbolic metric and pleated along measured laminations.\\

Later Krasnov and Schlenker (\cite{Schlenker-Krasnov}) introduced another parameterisation of $\mathcal{GH}(S)$ by the cotangent bundle over $\Teich(S)$, which is what inspired our construction. Let us recall it briefly here. Let $M$ be a GHMC anti-de Sitter manifold. It is well-known that $M$ contains a unique embedded maximal surface $S$ (\cite{foliationCMC}). Lifting $S$ to $\AdS_{3}$, we obtain an equivariant maximal embedding of $\h^{2}$ into $\AdS_{3}$, which is completely determined (up to global isometries of $\AdS_{3}$) by its induced metric and a holomorphic quadratic differential. By equivariance, these define a Riemannian metric $I$ and a holomorphic quadratic differential $q$ on $S$. We can thus define a map
\begin{align*}
	\Psi: \mathcal{GH}(S) &\rightarrow T^{*}\Teich(S) \\
			M &\mapsto (h,q)
\end{align*}
associating to a GHMC anti-de Sitter structure the unique hyperbolic metric in the conformal class of $I$ and the holomorphic quadratic differential $q$.\\

In order to prove that $\Psi$ is a homeomorphism, Krasnov and Schlenker (\cite{Schlenker-Krasnov}) found an explicit inverse. They showed that, given a hyperbolic metric $h$ and a quadratic differential $q$ that is holomorphic for the complex structure compatible with $h$, it is always possible to find a smooth map $v:S\rightarrow \R$ such that $I=2e^{2v}h$ and $B=I^{-1}\Re(2q)$ are the induced metric and the shape operator of a maximal surface embedded in a GHMC anti-de Sitter manifold. This is accomplished by noticing that the Codazzi equation for $B$ is trivially satisfied since $q$ is holomorphic, and thus it is sufficient to find $v$ so that the Gauss equation holds. Now,
\[
	\det(B)=\det(e^{-2v}(2h)^{-1}\Re(q))=e^{-4v}\det((2h)^{-1}\Re(2q))=-e^{-4v}\|q\|_{h}^{2}
\]
and
\[
	K_{I}=e^{-2v}(K_{2h}-\Delta_{2h}v)=\frac{1}{2}e^{-2v}(K_{h}-\Delta_{h}v)
\]
hence the Gauss equation translates into the quasi-linear PDE
\begin{equation}\label{eq:PDE}
	\frac{1}{2}\Delta_{h}v=e^{2v}-e^{-2v}\|q\|_{h}^{2}+\frac{1}{2}K_{h} \ .
\end{equation}

They proved existence and uniqueness of the solution to Equation (\ref{eq:PDE}) on closed surfaces and on surfaces with punctures, when $q$ has pole sigularities of order at most $1$ at the punctures. In Section \ref{sec:existence_max}, we will extend this result for meromorphic quadratic differentials on $\Sigma$ with poles of order at most $2$ at the punctures and describe the geometry of the associated maximal surface.

\subsubsection{Relation between the two parameterisations}\label{sec:relpara} The theory of harmonic maps between hyperbolic surfaces provides a bridge between the two parameterisations of $\mathcal{GH}(S)$. Let $M$ be a GHMC anti-de Sitter manifold with holonomy $\rho=(\rho_{l}, \rho_{r})$ and let $S$ be the unique maximal surface embedded in $M$. Lifting to the universal cover, the Gauss map $G:\tilde{S} \rightarrow \h^{2}\times \h^{2}$ provides a couple of $(\rho_{r}, \rho_{l})$-equivariant harmonic maps with Hopf differentials $\pm iq$, where $\Re(q)$ is the second fundamental form of $S$ (\cite{Schlenker-Krasnov}, \cite{TambuCMC}). Denoting with $\pi_{l}$ and $\pi_{r}$ the projections onto the left and right factor, the metrics
\[
	(G\circ\pi_{l})^{*}g_{\h^{2}}  \ \ \ \text{and} \ \ \ (G\circ \pi_{r})^{*}g_{\h^{2}}
\]
descend to hyperbolic metrics $h_{l}$ and $h_{r}$ on $S$ with holonomy $\rho_{l}$ and $\rho_{r}$, respectively. 

\begin{remark}The same picture holds for GHMC anti-de Sitter manifolds with particles: their deformation space is parameterised by a couple of hyperbolic metrics on $\Sigma$ with cone singularities of angle $\theta$ less than $\pi$ at the punctures, or equivalently by the vector bundle over $\Teich_{\theta}(\Sigma)$ of meromorphic quadratic differentials on $\Sigma$ with at most simple poles at the punctures. See also \cite{Tambu_Qiyu}, \cite{toulisse}.
\end{remark}

\subsection{Meromorphic quadratic differentials}
Suppose that $\Sigma$ is endowed with a complex structure. A meromorphic quadratic differential $q$ on $\Sigma$ is a $(2,0)$-tensor, locally of the form $q(z)dz^{2}$, where $q(z)$ is a meromorphic function with poles at the punctures $\{p_{1}, \dots, p_{k}\}$. In this paper, we are interested in meromorphic quadratic differentials with poles of order at most $2$ at the punctures, which we call \emph{regular}. This means that in a local coordinate around the puncture we can write
\[
	q(z)dz^{2}= \frac{R}{z^{2}}dz^{2}(1+O(z))
\]
for some $R \in \C$, called the residue of the quadratic differential (\cite[Chapter III]{Strebel}). Notice that $R$ is indipendent from the choice of the local coordinate. 

\begin{remark}The terminology \emph{regular} is not standard, and it is borrowed from the theory of parabolic Higgs bundles. Namely, $\dPSL$-parabolic Higgs bundles over $\Sigma$ associated to a meromorphic quadratic differential with poles of order at most $2$ at the punctures produce \emph{regular} singularities, as defined in \cite{Simpson_irregular}.
\end{remark}

By Riemann-Roch, the complex vector space of meromorphic quadratic differentials with poles at $\{p_{1}, \dots, p_{k}\}$ of order at most $2$ has real dimension $3|\chi(\overline{\Sigma})|+4k$. \\

A quadratic differential $q$ induces a singular flat metric $|q|$, that in local coordinates is written as $|q|=|q(z)||dz|^{2}$. The metric has cone singularities of angle $\pi(m+2)$ at the zeros of order $m$ of $q$, and of angle $\pi$ at a first order pole. Instead, the metric is complete in a neighbourhood of a second order pole.\\

\section{Finding a maximal surface}\label{sec:existence_max}
In the next three sections we are going to construct globally hyperbolic anti-de Sitter structures on $\Sigma \times \R$ starting from the data of the conformal structure induced by a complete hyperbolic metric $h$ on $\Sigma$ of finite area and a regular meromorphic quadratic differential $q$. We first find a complete equivariant maximal embedding of $\widetilde{\Sigma}$ into $\AdS_{3}$ with induced metric $I=2e^{2v}h$ and second fundamental form $II=\Re(2q)$. We then describe its boundary at infinity and prove that $\pi_{1}(\Sigma)$ acts by isometries and properly discontinously on its domain of dependence, thus inducing a globally hyperbolic anti-de Sitter structure on the quotient $\Sigma \times \R$. Moreover, we show how to determine the holonomy along peripheral curves in terms of the residue at the poles.\\

Let $h\in \mathcal{T}(\Sigma)$ be a complete hyperbolic metric of finite area on $\Sigma$ and let $q$ be a meromorphic quadratic differential with poles of order $2$ at the punctures $\{p_{j}\}_{j=1}^{k}$ with complex residues $\{R_{j}\}_{j=1}^{k}$ respectively. Recall that finding an equivariant maximal conformal embedding of $\widetilde{\Sigma}$ into $\AdS_{3}$ is equivalent to finding a solution of the quasi-linear PDE (Section \ref{sec:maxAdS})
\begin{equation}\label{eq:PDE2}
	\frac{1}{2}\Delta_{h}=e^{2v}-e^{-2v}\|q\|^{2}_{h}+\frac{1}{2}K_{h} \ .
\end{equation}
Notice that this equation is invariant under a conformal change of the metric $h$, in the sense that if $g$ is in the same conformal class as $h$ and $2e^{2v}h=2e^{2u}g$, then $u$ satisfies the differential equation
\begin{equation}\label{eq:PDE3}
	\frac{1}{2}\Delta_{g}=e^{2u}-e^{-2u}\|q\|^{2}_{g}+\frac{1}{2}K_{g} \ .
\end{equation}
Thus, we first want to choose an underlying complete metric $g$ conformal to $h$ such that
\[
	1-\|q\|_{g}^{2}+\frac{1}{2}K_{g} \to 0  \ \ \ \text{at $p_{i}$} \ ,
\]
so that  $u=0$ is an approximate solution to Equation (\ref{eq:PDE3}) near the punctures. We first describe how to choose the metric $g$ in a neighbourhood of the punctures. To this aim we distinguish two cases:
\begin{enumerate}[a)]
	\item if at a puncture $p_{i}$ the complex residue is $R_{i}=0$, we consider the metric $\frac{1}{2}h$. In fact, in a neighbourhood $U_{i}$ of $p_{i}$, we can find a local coordinate $z_{i}$ so that
\[
	h=\frac{4}{|z_{i}|^{2}(\log(|z_{i}|^{2}))^{2}}|dz_{i}|^{2}
\]
and it can be easily verified that $\|q\|_{g}^{2}=o(1)$ for $z_{i} \to 0$;
	\item otherwise we choose, on a neighbourhood $U_{i}$ of the puncture $p_{i}$, the flat metric 
\[
	g=\frac{|R_{i}|}{|z_{i}|^{2}}|dz_{i}|^{2}
\]
induced by the leading term of the quadratic differential. In this case, we will have $\|q\|_{g}^{2}\propto 1$ on $U_{i}$.
\end{enumerate}
We then define $g$ on all $\Sigma$ by smoothly interpolating on annular neighbourhoods of $U_{i}$ between the metric $\frac{1}{2}h$ and the metrics described above. More precisely, around a puncture $p_{i}$ where the complex residue $R_{i}$ does not vanish, we can find a complex coordinate $z_{i}$ and radii $c_{i}<C_{i}$ so that
\begin{equation}\label{eq:metric_coord}
	g_{|_{U_{i}}}=\begin{cases}
		\frac{|R|}{|z_{i}|^{2}}|dz_{i}|^{2} \ \ \ \ \ \ \ \ \ \ \ \ \ \ \ \ \ \ \ \ \ \ \text{for $|z_{i}|<c_{i}$} \\
		e^{v_{i}}|dz_{i}|^{2} \ \ \ \ \ \ \  \ \ \ \ \ \ \ \ \ \ \ \ \ \ \ \ \text{for $c_{i}\leq |z_{i}|\leq C_{i}$}\\
		\frac{2}{|z_{i}|^{2}(\log(|z_{i}|^{2}))^{2}}|dz_{i}|^{2} \ \ \ \ \ \ \ \ \ \text{for $|z_{i}|>C_{i}$}
		\end{cases} 
\end{equation}
for some smooth interpolating functions $v_{i}$. Moreover, we can require that all the zeros of $q$ be outside $U_{i}$ and that there exist $\delta_{i}>0$ such that $\|q\|_{g}^{2}\geq \delta_{i}$ on $U_{i}$. This is possible because $\|q(z_{i})\|_{g}^{2}\to 1$ when $z_{i}\to 0$.\\

We can now find a solution to Equation (\ref{eq:PDE3}) using the method of barriers.
\begin{prop}\label{prop:existence_sol}There exists a bounded smooth function $u:\Sigma \rightarrow \R$ satisfying
\begin{equation}\label{eq:PDEproof}
	\frac{1}{2}\Delta_{g}u=e^{2u}-e^{-2u}\|q\|_{g}^{2}+\frac{1}{2}K_{g} \ .
\end{equation}
\end{prop}
\begin{proof} Let $F(u,x)=e^{2u}-e^{-2u}\|q\|_{g}^{2}+\frac{1}{2}K_{g}$. Since $F$ is an increasing function of $u$, the existence of a solution to Equation (\ref{eq:PDEproof}) is guaranteed (\cite[Theorem 9]{wan}) by the existence of two continuous functions $u^{\pm}:\Sigma \rightarrow \R$ such that 
\[
	\Delta u^{+}\leq F(u^{+},x) \ , \ \ \ \Delta u^{-}\geq F(u^{-},x) \ \ \ \text{and} \ \ \ u^{-}\leq u^{+} \ .
\]
Let us start with the supersolution $u^{+}$. We consider a function $f:\Sigma \rightarrow \R$ with the following properties:
\begin{itemize}
	\item $f(z_{i})=|z_{i}|^{2\alpha_{i}}$ on the neighbourhood $\{|z_{i}|<c_{i}\}$ (see (\ref{eq:metric_coord})) of the puncture $p_{i}$ with non-vanishing complex residue;
	\item $f$ is a positive constant on a neighbourhood of the punctures with zero residue;
	\item $f$ is smooth and positive everywhere on $\Sigma$.
\end{itemize}
We then define $u^{+}=\beta f$ for some $\beta \in \R$. We claim that it is possible to choose $\beta>0$ sufficiently large and $\alpha_{i}>0$ small enough so that $u^{+}$ is a supersolution. It is clear that $u^{+}$ is a supersolution for every choice of $\beta$ sufficiently large on the neighbourhoods of the punctures with vanishing residue, because $f$ is constant. For the other cases, on the balls $\{|z_{i}|<c_{i}\}\subset U_{i}$ we compute
\begin{align*}
	&\frac{1}{2}\Delta_{g}(\beta|z_{i}|^{2\alpha_{i}})-e^{2\beta|z_{i}|^{2\alpha_{i}}}+e^{-2\beta|z_{i}|^{2\alpha_{i}}}\|q\|_{g}^{2}+\frac{1}{2}K_{g}\\
	=&\frac{1}{2}\beta\alpha_{i}^{2}\frac{|z_{i}|^{2\alpha_{i}}}{|R_{i}|}-e^{2\beta|z_{i}|^{2\alpha_{i}}}+e^{-2\beta|z_{i}|^{2\alpha_{i}}}(1+O(|z|))\\
	=&\left(\frac{\alpha_{i}^{2}}{2|R_{i}|}-2\right)u^{+}+(e^{-2u^{+}}-e^{2u^{+}}+2u^{+})+e^{-2\beta|z_{i}|^{2\alpha_{i}}}(1+O(|z|))
\end{align*}
and we notice that the term in the middle is always non-positive and we can choose $\alpha_{i}$ small enough and $\beta$ large enough so that the sum of the first and last term is negative. Therefore $u^{+}$ is a supersolution on $U_{i}$ for every $\alpha_{i}<\alpha_{0}$ and $\beta>\beta_{0}$. Outside these balls, we do not have control on the curvature of $g$ and on the Laplacian of $f$, but knowing that they are bounded, we can increase $\beta$ so that
\[
	\frac{\beta}{2} \Delta_{g}f-e^{\beta f}+e^{-2\beta f }\|q\|_{g}^{2}-\frac{1}{2}K_{g} \leq 0
\]
because $e^{\beta f}$ grows the fastest when $\beta \to +\infty$. This proves that $u^{+}$ is a supersolution everywhere on $\Sigma$.\\
\indent As for the subsolution, let us first consider a neighbourhood of the puncture $p_{i}$ with residue $R_{i}\neq 0$. On the ball $\{|z_{i}|<c_{i}\}$ a similar computation as above shows that it is possible to choose $\alpha_{i}>0$ small so that $w(z_{i})=-\beta_{i} |z_{i}|^{2\alpha_{i}}$ is a subsolution. Namely, 
\begin{align*}
	&\frac{1}{2}\Delta_{g}(w)-e^{2w}+e^{-2w}\|q\|_{g}^{2}-\frac{1}{2}K_{g}\\
	=&-\frac{1}{2}\beta_{i}\alpha_{i}^{2}\frac{|z_{i}|^{2\alpha_{i}}}{|R_{i}|}-e^{-2\beta_{i}|z_{i}|^{2\alpha_{i}}}+e^{2\beta_{i}|z_{i}|^{2\alpha_{i}}}(1+O(|z|))\\
	=&\left(\frac{\alpha_{i}^{2}}{2|R_{i}|}-2\right)w+(e^{-2w}-e^{2w}+2w)+e^{-2w}(1+O(|z|))
\end{align*}
is the sum of three positive terms for every $\beta_{i}>0$ and for $\alpha_{i}$ sufficiently small. On the annuli $\{c_{i}\leq |z_{i}| \leq C_{i}\}$, we do not have control on the curvature of $g$ (see (\ref{eq:metric_coord})), but we know that there exists $\delta_{i}>0$ so that $\|q\|_{g}^{2}\geq \delta_{i}$. Therefore, we can choose $\beta_{i}>0$ large enough so that the term $e^{-2w}\|q\|_{g}^{2}$ becomes dominant. In this way $w$ is a subsolution on the bigger balls $\{|z_{i}|\leq C_{i}\}$ and we can suppose that it takes a fixed value $-B<0$ on the boundary of each of those balls. We then define a function on the entire $\Sigma$ by putting
\[
	u^{-}=\begin{cases}
		-\beta_{i}|z_{i}|^{2\alpha_{i}} \ \ \ \ \text{on $\{|z_{i}|\leq C_{i}\}\subset U_{i}$} \\ 
		-B \ \ \ \ \ \ \ \ \ \ \ \ \text{elsewhere} \ . 
	\end{cases}
\]
It can be easily checked that the constant function $-B<0$ is always a subsolution when $g$ has constant curvature $-2$, hence $u^{-}$ is a subsolution on all $\Sigma$. 
\end{proof}

Uniqueness of the solution to Equation (\ref{eq:PDEproof}) follows from the Cheng and Yau's maximum principle (\cite{chengyaumaxprinciple}). 
\begin{prop}\label{prop:uniqueness}There exists a unique bounded solution to Equation (\ref{eq:PDEproof}) for a given complete metric $g$ and regular meromorphic quadratic differential $q$.
\end{prop}
\begin{proof} Suppose $u$ and $u'$ are two bounded solution to Equation (\ref{eq:PDEproof}). The difference $\eta=u-u'$ satisfies 
\[
	\frac{1}{2}\Delta_{g}\eta=F(u,x)-F(u',x) 
\]
where $F(u,x)=e^{2u}-e^{-2u}\|q\|_{g}^{2}+\frac{1}{2}K_{g}$. Since $u$ and $u'$ are bounded, there exists a positive constant $C$ such that
\[
	\Delta_{g}\eta\geq C\eta \ .
\]
Since $g$ is complete and has bounded curvature, Cheng and Yau's result implies that there exists a sequence $x_{n}\in \Sigma$ such that
\[
	\Delta_{g}\eta(x_{n})\leq \frac{1}{n}  \ \ \ \text{and} \ \ \ \eta(x_{n})\geq M-\frac{1}{n}
\]
where $M=\sup_{\Sigma}\eta$. Therefore, the chain of inequalities
\[
	\frac{1}{n}\geq \Delta_{g}\eta(x_{n})\geq C\eta(x_{n})\geq C\left(M-\frac{1}{n}\right)
\]
implies that $M\leq 0$ and $\eta \leq 0$. By switching the roles of $u$ and $u'$ we obtain similarly that $\eta \geq 0$, hence $u=u'$.
\end{proof}		

\begin{teo}\label{cor:riassunto} For any complete hyperbolic metric $h$ on $\Sigma$ of finite area and for any regular meromorphic quadratic differential $q$ on $\Sigma$ there exists a unique equivariant maximal embedding $\tilde{\sigma}:\tilde{\Sigma}\rightarrow \AdS_{3}$ with induced metric $I$ conformal to $h$ and second fundamental form $II=\Re(2q)$. Moreover, $I$ is complete and the principal curvatures are in $(-1,1)$.
\end{teo}
\begin{proof}Existence and uniqueness of such maximal embedding follows from the above discussion. Moreover, the induced metric can be written as $I=2e^{2u}g$, where $g$ is the metric defined at the beginning of Section \ref{sec:existence_max} and $u$ is the solution to Equation (\ref{eq:PDE3}), hence it is complete because $g$ is complete and $u$ is bounded. \\
\indent Let $\lambda$ be the positive principal curvature of the maximal embedding. By definition of $q$, we have
\[
	-\lambda^{2}=\det(B)=e^{-4u}\|q\|_{g}^{2} \to \begin{cases} 0  \ \ \ \text{if the residue vanishes at $p_{i}$} \\
								    1  \ \ \  \text{otherwise}
							\end{cases}
\]
thus $\lambda$ is bounded. A classical fact about maximal surfaces in anti-de Sitter space (\cite[Lemma 3.11]{Schlenker-Krasnov}) implies that $\lambda \in [0,1)$.    
\end{proof}

\section{Description of the holonomy representation}\label{sec:holonomy}
The equivariant maximal embedding $\tilde{\sigma}:\tilde{\Sigma}\rightarrow \AdS_{3}$ comes with a representation $\rho:\pi_{1}(\Sigma)\rightarrow \Pp\SO_{0}(2,2)$ such that
\[
   \tilde{\sigma}(\gamma \cdot x)=\rho(\gamma)\tilde{\sigma}(x) \ \ \ \forall x \in \tilde{\Sigma}  \ \ \forall \gamma \in \pi_{1}(\tilde{\Sigma}) \ .
\]
Identifying $\Pp\SO_{0}(2,2)$ with $\dPSL$, $\rho$ determines and is determined by a couple of representations $\rho_{l,r}:\pi_{1}(\Sigma) \rightarrow \PSL(2,\R)$. As explained in Section \ref{sec:relpara}, these are holonomies of hyperbolic structures on $\Sigma$ and can be described explicitly in terms of the data of the maximal embedding $\tilde{\sigma}$. In particular, we are able to compute the holonomy of the peripheral curves in terms of the complex residues. \\

Let us first describe the right setting to perform this computation. We identify the universal cover of $\Sigma$ with the upper half-plane $\h^{2}=\{w=x+iy \in \C \ | \ y>0\}$. Each puncture of $\Sigma$ corresponds to a parabolic element in $\PSL(2,\R)$, which is conjugate to $\gamma(w)=w+2\pi$. For a given puncture $p$ (we suppress the indices in this discussion), we choose a local conformal coordinate $z$ and identify a neighbourhood of $p$ in $\Sigma$ with the punctured disk $D_{0}=\{z \in \C \ | \ 0<|z|<\epsilon\}$. Let $\zeta:\h^{2} \rightarrow D_{0}$ be the covering map $\zeta(w)=e^{iw}$. The punctured disk $D_{0}$ lifts to the strip $\{w \ | \ y>-\log(\epsilon) \}$ and the map $\gamma$ generates the deck transformations for the covering $\zeta$. Moreover, for every $y>>0$, each horizontal line segment $\gamma_{y}(t)=(t,y)\in \h^{2}$ for $t \in [0,2\pi]$ projects under $\zeta$ to a peripheral curve around the puncture. \\

From the work of Krasnov and Schlenker (\cite{Schlenker-Krasnov}), we know that the left representation $\rho_{l}:\pi_{1}(\Sigma)\rightarrow \PSL(2,\R)$ is the holonomy of the hyperbolic metric on $\Sigma$
\[
	h_{l}=I(E+JB, E+JB) \ ,
\]
where $E:T\Sigma \rightarrow T\Sigma$ is the identity operator, $B$ is the shape operator of the maximal embedding $\tilde{\sigma}$ and $J$ is the complex structure compatible with the induced metric. We need to understand the nature of this metric around a puncture. In the $w$-coordinate, the quadratic differential $q$ is given by
\[
	\zeta^{*}q=-R(1+O(e^{-y}))dw^{2} \ 
\]
thus, recalling that $I=2e^{2u}g$ and $B=I^{-1}\Re(2q)$, we can write
\[
	\zeta^{*}h_{l}(w)=2e^{2u}(1+\lambda^{2})g-2iR(1+O(e^{-y}))dw^{2}+2i\bar{R}(1+O(e^{-y}))d\bar{w}^{2} \ .
\]
This is the local expression of a hyperbolic metric with geodesic boundary (\cite[p. 516]{Mike_infenergy}) and the length of the boundary curve can be computed as
\[
	\lim_{y \to +\infty}\ell_{h_{l}}(\zeta(\gamma_{y})) \ .
\]
In this case we have
\begin{align*}
	\ell_{h_{l}}(\zeta(\gamma_{y}))&=\int_{0}^{2\pi}\sqrt{2e^{2u}(1+\lambda^{2})|R|+4\Imm(R)(1+O(e^{-y}))} dt \\
	 &\xrightarrow{y\to +\infty} 4\pi\sqrt{|R|+\Imm(R)}
\end{align*}
by Theorem \ref{cor:riassunto}. Notice that if $R=0$ or $\Re(R)=0$ and $\Imm(R)<0$, the length of the boundary vanishes, thus the corresponding puncture is a cusp end. A similar reasoning can be applied also to the right-representation $\rho_{r}:\pi_{1}(\Sigma) \rightarrow \PSL(2,\R)$ which is the holonomy of the hyperbolic metric
\[
	h_{r}=I(E-JB, E-JB) \ ,
\]
and leads to the following result:

\begin{prop}\label{prop:holonomy} Let $\gamma \in \pi_{1}(\Sigma)$ be a peripheral curve around a puncture $p$ with complex residue $R$. Then 
\begin{enumerate}[a)]
	\item if $\Re(R)\neq 0$, then $\rho_{l}(\gamma)$ and $\rho_{r}(\gamma)$ are both hyperbolic with translation length $\ell_{l}(\gamma)=4\pi\sqrt{|R|+\Imm(R)}$ and $\ell_{r}(\gamma)=4\pi\sqrt{|R|-\Imm(R)}$;
	\item if $\Re(R)=0$ and $\Imm(R)>0$, then $\rho_{r}(\gamma)$ is parabolic and $\rho_{l}(\gamma)$ is hyperbolic with translation length $\ell_{l}(\gamma)=4\pi\sqrt{2\Imm(R)}$;
	\item if $\Re(R)=0$ and $\Imm(R)<0$, then $\rho_{l}(\gamma)$ is parabolic and $\rho_{r}(\gamma)$ is hyperbolic with translation length $\ell_{r}(\gamma)=4\pi\sqrt{-2\Imm(R)}$;
	\item if $R=0$, then $\rho_{l}(\gamma)$ and $\rho_{r}(\gamma)$ are both parabolic. 
\end{enumerate}
\end{prop}

\begin{remark} Later we will recover Proposition \ref{prop:holonomy} by different methods, without using the theory developed in \cite{Schlenker-Krasnov}. However, our techniques will not distinguish between left- and right- representations. 
\end{remark}

\section{Description of the domain of dependence}\label{sec:boundary}

Since the maximal surface is complete, its boundary at infinity is a locally achronal curve $\Gamma$ and determines a domain of dependence $\mathcal{D}(\Gamma)\subset \AdS_{3}$ by considering points whose dual space-like planes are disjoint from $\Gamma$. Let us point out the relations between $\Gamma$ and the limit set of the representation $\rho=(\rho_{l},\rho_{r})$. Recall that the boundary at infinity of $\AdS_{3}$ is identified with $\R\Pp^{1}\times \R\Pp^{1}$ and $\rho_{l}$ (resp. $\rho_{r}$) acts on the left (resp. right) factor by projective tranformations. Given an element $\gamma \in \pi_{1}(\Sigma)$, we denote with $x^{\pm}_{\bullet}(\gamma)$ the attractive and repulsive fixed points of $\rho_{\bullet}$, with the convention that $x^{+}_{\bullet}(\gamma)=x^{-}_{\bullet}(\gamma)$ if $\rho_{\bullet}(\gamma)$ is parabolic. These define four points (possibly coincident) on the boundary at infinity of $\AdS_{3}$:
\begin{align*}
	x^{++}(\rho(\gamma))&=(x^{+}_{l}(\gamma), x^{+}_{r}(\gamma)) \ \ \ \ \ \ \ \ \ \ x^{+-}(\rho(\gamma))=(x^{+}_{l}(\gamma), x^{-}_{r}(\gamma)) \\
	x^{--}(\rho(\gamma))&=(x^{-}_{l}(\gamma), x^{-}_{r}(\gamma)) \ \ \ \ \ \ \ \ \ \  x^{-+}(\rho(\gamma))=(x^{-}_{l}(\gamma), x^{+}_{r}(\gamma)) \ .
\end{align*}
It follows immediately from the definition that 
\[
	\lim_{n\to +\infty} \rho(\gamma)^{n}\cdot x= x^{++}(\rho(\gamma))
\]
for every $x \in \partial_{\infty}\AdS_{3} \setminus \{x^{+-}(\rho(\gamma)), x^{-+}(\rho(\gamma)), x^{--}(\rho(\gamma))\}$. Therefore, the limit set
\[
	\Lambda_{\rho}=\overline{\{( x^{++}(\rho(\gamma))\in \partial_{\infty}\AdS_{3} \ | \ \gamma \in \pi_{1}(\Sigma) \}}
\]
is the smallest closed $\rho(\pi_{1}(\Sigma))$-invariant subset in the boundary at infinity of anti-de Sitter space. Since the maximal surface found in Section \ref{sec:existence_max} is $\rho(\pi_{1}(\Sigma))$-invariant, its boundary at infinity must contain $\Lambda_{\rho}$. We have thus proved:

\begin{lemma}\label{prop:limitset}Let $\rho:\pi_{1}(\Sigma)\rightarrow \Pp\SO_{0}(2,2)$ be the holonomy representation of a maximal embedding $\tilde{\sigma}:\tilde{\Sigma}\rightarrow \AdS_{3}$. Then the limit set $\Lambda_{\rho}$ of $\rho$ is contained in the boundary at infinity of $\tilde{\sigma}(\tilde{\Sigma})$.
\end{lemma}

We notice in particular that if $\rho_{r}$ and $\rho_{l}$ are the holonomies of complete hyperbolic metrics on $\Sigma$, then $\Lambda_{\rho}$ is a topological circle, hence the boundary at infinity coincides with $\Lambda_{\rho}$. Otherwise, $\Lambda_{\rho}$ is a Cantor set (\cite[Proposition 7.2]{bsk_multiblack}) and we need to describe how to complete $\Lambda_{\rho}$ to the whole boundary at infinity of the maximal surface. If for a peripheral element $\gamma \in \pi_{1}(\Sigma)$, we have that $\rho_{l}(\gamma)$ is hyperbolic and $\rho_{r}(\gamma)$ is parabolic (or viceversa), then, since $x^{+}_{r}(\gamma)=x^{-}_{r}(\gamma)$ (or $x^{+}_{l}(\gamma)=x^{-}_{l}(\gamma)$), the limit set contains the points $x^{++}(\rho(\gamma))$ and $x^{--}(\rho(\gamma))$ that are causally related. Hence, by Lemma \ref{lm:boundary} the boundary at infinity of the maximal surface contains the whole light-like segment joining them.\\  

We are thus left to understand the boundary at infinity when the holonomies of a peripheral element are both hyperbolic, equivalently, when the residue of the quadratic differential at the corresponding puncture is non-zero and not purely-imaginary. Recall that we modelled a neighbourhood of a puncture of $\Sigma$ by a punctured disk $D_{0}$ and we defined the covering map $\zeta(w)=e^{iw}$ from the upper-half plane $\h^{2}$ to $D_{0}$. We introduce another change of coordinates: choose $\xi \in \C$ so that $\xi^{2}=-R\neq 0$ and define $\omega=\xi w=\eta+i\tau$, so that
\[
	q(\omega)=d\omega^{2}(1+o(1)) \ \ \ \text{and} \ \ \ g(\omega)=|d\omega|^{2} \ .
\]
Since the conformal factor $u$ tends to $0$ at the puncture, the embedding data of the maximal surface can be approximated in the $\omega$-coordinate by the constant quadratic differential $d\omega^{2}$ and the constant flat metric $|d\omega|^{2}$ close to the puncture. This suggests that the maximal embedding $\tilde{\sigma}$ should look like the horospherical surface in such a neighbourhood. The rest of the section builds up on this intuition.

\subsection{The frame field of a maximal embedding} Let us consider $\R^{4}\subset \C^{4}$ and extend the $\R$-bilinear form of signature $(2,2)$ to the hermitian product on $\C^{4}$ given by
\[
	\langle z,w\rangle=z_{1}\bar{w}_{1}+z_{2}\bar{w}_{2}-z_{3}\bar{w}_{3}-z_{4}\bar{w}_{4} \ .
\]
Given a maximal conformal embedding $\tilde{\sigma}:\h^{2} \rightarrow \AdS_{3}$, with a slight abuse of notation, we still denote with $\tilde{\sigma}:\h^{2}\rightarrow \widehat{\AdS_{3}} \subset \C^{4}$ one of its lifts. Let $N$ be the unit normal vector field such that $\{\tilde{\sigma}_{w}, \tilde{\sigma}_{\bar{w}}, N, \tilde{\sigma}\}$ is an oriented frame in $\C^{4}$. We define
\[
	q=\langle N_{w}, \tilde{\sigma}_{\bar{w}} \rangle \ .
\]
The embedding being maximal implies that $q$ is a holomorphic quadratic differential on $\h^{2}$. Since the embedding is conformal, we can define a function $\phi:\h^{2} \rightarrow \R$ such that
\[
	\langle \tilde{\sigma}_{w}, \tilde{\sigma}_{w} \rangle= \langle \tilde{\sigma}_{\bar{w}}, \tilde{\sigma}_{\bar{w}} \rangle=e^{2\phi} \ .
\]
These are related to the embedding data of $\tilde{\sigma}$ as follows: the induced metric on $\tilde{\sigma}(\h^{2})$ is $I=2e^{2\phi}|dw|^{2}$ and the second fundamental form is $II=\Re(2q)$. The vectors 
\[
	v_{1}=\frac{\tilde{\sigma}_{w}}{e^{\phi}} \ \ \ v_{2}=\frac{\tilde{\sigma}_{\bar{w}}}{e^{\phi}} \ \ \ N, \ \ \ \text{and} \ \ \ \tilde{\sigma}
\]
give a unitary frame of $(\C^{4}, \langle \cdot, \cdot \rangle)$ at every point $w\in \h^{2}$. Taking the derivatives of the fundamental relations
\[
	\langle N,N\rangle=\langle \tilde{\sigma},\tilde{\sigma} \rangle=-1 \ \ \langle v_{j}, N \rangle=\langle v_{j},\tilde{\sigma}\rangle=0 \ \ \langle N_{z}, \tilde{\sigma}_{\bar{w}} \rangle=q \ \ \langle v_{j},v_{j}\rangle=1
\]
one deduces that
\[
	N_{\bar{w}}=e^{-\phi}\bar{q}v_{1} \ \ \ \overline{\partial}v_{1}=-\phi_{\bar{w}}v_{1}+e^{\phi}\tilde{\sigma} \ \ \ \text{and} \ \ \ \overline{\partial}v_{2}=\phi_{\bar{w}}v_{2}+\bar{q}e^{-\phi}N \ .
\]
Therefore, the pull-back of the Levi-Civita connection $\nabla$ of $(\C^{4}, \langle \cdot, \cdot, \rangle)$ via $\tilde{\sigma}$ can be written in the frame $\{v_{1},v_{2},N,\tilde{\sigma}\}$ as
\begin{equation}\label{eq:framefield}
	\tilde{\sigma}^{*}\nabla=Vd\bar{w}+Udw=
		\begin{pmatrix} -\phi_{\bar{w}} & 0 & e^{-\phi}\bar{q} & 0 \\
				0 & \phi_{\bar{w}} & 0 & e^{\phi} \\ 
				0 & e^{-\phi}\bar{q} & 0 & 0 \\
				e^{\phi} & 0 & 0 & 0
		\end{pmatrix}d\bar{w}+ \begin{pmatrix} \phi_{w} & 0 & 0 & e^{\phi} \\
							0 & -\phi_{w} & qe^{-\phi} & 0 \\
							qe^{-\phi} & 0 & 0 & 0 \\
							0 & e^{\phi} & 0 & 0 
				       \end{pmatrix}dw \ .
\end{equation}
Notice that the flatness of $\tilde{\sigma}^{*}\nabla$ is equivalent to $\phi$ being a solution of the PDE
\[
	\frac{1}{2}\Delta \phi= e^{2\phi}-e^{-2\phi}|q|^{2}
\]
which coincides with Equation (\ref{eq:PDE}) when the background metric is flat. \\

Viceversa, if a holomorphic quadratic differential $q$ and a solution $\phi$ of the above equation are given, the $1$-form $Vd\bar{w}+Udw$ can be integrated to a map $F:\h^{2}\rightarrow \SL(4,\C)$, which is the frame field of a maximal embedding into $\AdS_{3}$ with induced metric $I=2e^{2\phi}|dw|^{2}$ and second fundamental form $II=\Re(2q)$. Moreover, this is unique once the initial conditions are fixed.

\subsection{The horospherical surface}The frame field can be written explicitly in the special case when $q$ is a constant holomorphic quadratic differential, and the associated maximal surface in $AdS_{3}$ appears in the literature as the horospherical surface (\cite{bon_schl}, \cite{seppimaximal}, \cite{TambuCMC}). See also \cite{Tambu_poly}.\\

Suppose $q=d\omega^{2}$ is a holomorphic quadratic differential defined on the complex plane $\C$. The corresponding solution to the flatness equation is then clearly $\phi=0$. The $1$-form becomes 
\[
	V_{0}d\bar{\omega}+U_{0}d\omega=\begin{pmatrix}
			0 & 0 & 1 & 0 \\
			0 & 0 & 0 & 1 \\
			0 & 1 & 0 & 0 \\
			1 & 0 & 0 & 0 
			\end{pmatrix} d\bar{\omega}+ \begin{pmatrix}
						0 & 0 & 0 & 1 \\
						0 & 0 & 1 & 0 \\
						1 & 0 & 0 & 0 \\
						0 & 1 & 0 & 0 
						\end{pmatrix}d\omega \ .
\]
The frame field of the horospherical surface is thus 
\[
	F_{0}(\omega)=A_{0}\exp(U_{0}\omega+V_{0}\bar{\omega}) \ ,
\]
for some constant matrix $A_{0} \in \SL(4, \C)$. For our convenience, we choose
\[
    A_{0}=\frac{1}{\sqrt{2}}\begin{pmatrix}
            1 & 1 & 0 & 0 \\
            -i & i & 0 & 0 \\
            0 & 0 & 1 & 1 \\
            0 & 0 & -1 & 1
            \end{pmatrix}
\]
A simple computation shows that the matrix $U_{0}\omega+V_{0}\bar{\omega}$ is diagonalisable by a constant unitary matrix $S$ so that
\[
	S^{-1}(U_{0}\omega+V_{0}\bar{\omega})S=\diag(2\Re(\omega), 2\Imm(\omega), -2\Re(\omega), -2\Imm(\omega)) \ .
\]
Therefore, we can write 
\[
	F_{0}(\omega)=A_{0}S\diag(e^{2\Re(\omega)}, e^{2\Imm(\omega)}, e^{-2\Re(\omega)}, e^{-2\Imm(\omega)})S^{-1} \ . 
\]
The resulting maximal embedding is given by the last column of $F_{0}(\omega)$, that is
\[
    \sigma_{0}=\frac{1}{\sqrt{2}}(\sinh(2\Re(\omega)), \sinh(2\Imm(\omega)), \cosh(2\Re(\omega)), \cosh(2\Imm(\omega)))^{t} \ .
\]
In particular, we can describe explicitly the boundary at infinity $\Delta$ of $\sigma_{0}$: it consists of four light-like segments as the following table shows.

\medskip
\begin{table}[!htb]
\begin{center}
\begin{tabular}{l l l}
\hline
\textbf{Direction $\theta$} & \textbf{Projective limit $v_\theta$ of $\sigma_{0}(te^{i\theta}+iy)$}\\
\hline
$\theta \in (-\tfrac{\pi}{4}, \tfrac{\pi}{4})$ & $v_\theta = [1,0,1,0]^{t}$\\
$\theta = \tfrac{\pi}{4}$ & $v_y=[1,s,1,s]^{t}$ for some $s(y) \in \R^{+}$\\
$\theta \in (\tfrac{\pi}{4}, \tfrac{3\pi}{4})$ & $v_\theta = [0,1,0,1]^{t}$\\
$\theta = \tfrac{3\pi}{4}$ & $v_y=[-s,1,s,1]^{t}$ for some $s(y) \in \R^{+}$\\
$\theta\in (\tfrac{3\pi}{4}, \tfrac{5\pi}{4})$  & $v_\theta = [-1,0,1,0]^{t}$\\
$\theta = \tfrac{5\pi}{4}$ & $v_y=[-1,-s,1,s]^{t}$ for some $s(y) \in \R^{+}$\\
$\theta\in (\tfrac{5\pi}{4}, \tfrac{7\pi}{4})$  & $v_\theta = [0,-1,0,1]^{t}$\\
$\theta = \tfrac{7\pi}{4}$ & $v_y=[s,-1,s,1]^{t}$ for some $s(y) \in \R^{+}$\\
\hline\\
\end{tabular}
\end{center}
\caption{Limits of the standard horospherical surface along rays}\label{table:1}
\end{table}

\noindent Moreover, in the boundary at infinity we can see two past-directed and two future-directed saw-teeth (Figure \ref{fig:sawteeth}).

\begin{figure}[htbp]
\centering
\includegraphics[height=6cm]{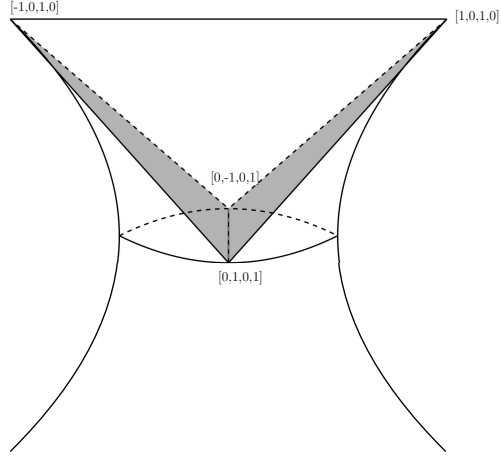}
\caption{\small{Saw-teeth in the boundary at infinity of the horopherical surface. Light-like planes bounding future-directed saw-teeth are highlighted.}}\label{fig:sawteeth}
\end{figure}

\subsection{Alternative way to determine the holonomy representation}\label{subsec:holonomy_alt}
In case $\tilde{\sigma}:\h^{2} \rightarrow \AdS_{3}$ is equivariant with respect to a representation $\rho:\pi_{1}(\Sigma)\rightarrow \Pp\SO_{0}(2,2)$, we can use the frame field equation (\ref{eq:framefield}) to compute the holonomy along a peripheral curve. Fix a base point $w_{0} \in \h^{2}$. For every deck transformation $\gamma \in \pi_{1}(\Sigma)$, that we think of as a holomorphic automorphism of $\h^{2}$, the uniqueness of the solution to the initial value problem implies that $F(\h^{2})=\gamma^{*}F(\h^{2})$, where the frame field pulls back under $\gamma$ to
\[
 \gamma^{*}F=\{\gamma'\tilde{\sigma}_{w}\circ \gamma, \overline{\gamma'}\tilde{\sigma}_{\bar{w}}\circ \gamma, N\circ \gamma, \tilde{\sigma}\circ \gamma\} \ .
\]
In particular, if $\gamma$ is a peripheral curve, we can represent the deck tranformation as $\gamma(w)=w+2\pi$. Thus the matrix $H_{\gamma}$ defined by
\[
	H_{\gamma}: \{ \tilde{\sigma}_{w}(w_{0}), \tilde{\sigma}_{\bar{w}}(w_{0}), N(w_{0}), \tilde{\sigma}(w_{0})\} \mapsto 
			\{ \tilde{\sigma}_{w}(\gamma(w_{0})), \tilde{\sigma}_{\bar{w}}(\gamma(w_{0})), N(\gamma(w_{0})), \tilde{\sigma}(\gamma(w_{0}))\} \ 
\]
is conjugated to a matrix in $\SO_{0}(2,2)$ and its projection to $\Pp\SO_{0}(2,2)$ gives the element $\rho(\gamma)$, up to conjugation. We remark that $H_{\gamma}$ acts on the right on frame fields, whereas $\rho(\gamma)$ acts on the left on column vectors of $\R^{4}$, hence $\rho(\gamma)$ and $H_{\gamma}$ differ by conjugation by the frame field at the base point $w_{0}$. In particular, if we consider the line segment $\gamma_{y}(x)=(x,y) \in \h^{2}$ for $x \in [0,2\pi]$, this projects to a closed peripheral curve on $\Sigma$ for $y>>0$, and the holonomy matrix $H_{y}$ along this path is equal to $\Phi_{y}(2\pi)$, where $\Phi_{y}$ solves the initial value problem
\begin{align*}
	\Phi^{-1}_{y}\frac{\partial \Phi_{y}}{\partial x}&=A_{y}=\begin{pmatrix} -i\phi_{y} & 0 & e^{-\phi}\bar{q} & e^{\phi} \\
							0 & i\phi_{y} & qe^{-\phi} & e^{\phi} \\
							qe^{-\phi} & e^{-\phi}\bar{q} & 0 & 0 \\
							e^{\phi} & e^{\phi} & 0 & 0 
				      \end{pmatrix}\\
			\Phi_{y}(0)&=\Id \ .
\end{align*}
Namely, the family of matrices $\Phi_{y}(x)$ along the path $\gamma_{y}$ obtained in this way satisfy
\[
	F(0,y)\Phi_{y}(x)=F(x,y) \ \ \ \text{for every $x \in [0,2\pi]$} \ .
\]
Since $\sigma^{*}\nabla$ is flat and the loops $\gamma_{y}$ are freely homotopic, all $H_{y}$ are conjugated in $\SL(4,\C)$. Moreover, by the theory of linear ODE with parameters (\cite{ODE}), we know that
\[
	\lim_{y \to +\infty}\Phi_{y}(2\pi)=e^{2\pi A}
\]
where 
\[
	A=\lim_{y \to +\infty}A_{y}=\begin{pmatrix} 0 & 0 & -\bar{R}|R|^{-\frac{1}{2}} & |R|^{\frac{1}{2}} \\
					            0 & 0 & -R|R|^{-\frac{1}{2}} & |R|^{\frac{1}{2}} \\
						    -R|R|^{-\frac{1}{2}} & -\bar{R}|R|^{-\frac{1}{2}} & 0 & 0 \\
						    |R|^{\frac{1}{2}} & |R|^{\frac{1}{2}} & 0 & 0 \\
					\end{pmatrix} \ .
\]
Therefore, we can conclude that all matrices $H_{y}$ have the same eigenvalues $e^{2\pi \lambda_{i}}$, where $\lambda_{i}$ are the eigenvalues of $A$. Those are the roots of the characteristic polynomial 
\[
	\chi_{A}(t)=t^{4}-4|R|t^{2}+4\Imm(R)^{2}
\]
that can be easily computed: 
\[
 \lambda_{1}=-\lambda_{3}=\sqrt{2(|R|+|\Re(R)|)} \ \ \ \text{and} \ \ \ \lambda_{2}=-\lambda_{4}=\sqrt{2(|R|-|\Re(R)|)} \ .
\]
Using the identification between $\Pp\SO_{0}(2,2)$ and $\dPSL$ (\cite[Section 2]{entropy}) one can recover the eigenvalues of the corresponding matrices in $\PSL(2,\R)$ and the results of Proposition \ref{prop:holonomy}, up to the ambiguity of choosing the left and right factor. This ambiguity comes from the fact that this method does not give any information on the position of the eigenvalues along the diagonals and not every permutation of the diagonal entries can be realised by conjugating with an element in $\SO_{0}(2,2)$. \\ 

The representation $\rho(\gamma)$ of a peripheral element $\gamma \in \pi_{1}(\Sigma)$ with base point $w_{0}=(0,y_{0}) \in \h^{2}$ can thus be computed as follows: for every $y>y_{0}$ consider the concatanation of the vertical path from $w_{0}$ to $(0, y)$ followed by $\gamma_{y}$ and then by the vertical path from $(2\pi, y)$ to $(2\pi, y_{0})$. If $M_{y}$ denotes the holonomy along the vertical path, then  
\[
	H_{y_{0}}=M_{y}H_{y}M_{y}^{-1} \ .
\]
Since the resulting path is in the same homotopy class as the peripheral element $\gamma$ for every $y>y_{0}$, we can conclude that
\[
	\rho(\gamma)p=\lim_{y \to +\infty}CM_{y}H_{y}M_{y}^{-1}C^{-1}p \ \ \ \text{for every $p \in \AdS_{3}\cup \partial_{\infty}\AdS_{3}$} 
\]
where $C$ is a constant matrix depending on the frame field at the base point $w_{0}$.

\subsection{Comparison with the horospherical surface}\label{subsec:comparison}We have now all the ingredients to describe the boundary at infinity of the maximal surface $\tilde{\sigma}(\tilde{\Sigma})$ in a neighbourhood of a puncture $p$ with non-vanishing and non-purely imaginary residue. Recall that we identified a neighbourhood of $p$ in $\Sigma$ with a punctured disk $D_{0}=\{z \in \C \ | \ 0< |z|<\epsilon \}$ and defined the covering map $\zeta:\h^{2} \rightarrow D_{0}$ given by $\zeta(w)=e^{iw}$. The punctured disk lifts to a strip $N=\{w \in \h^{2} \ | \ y>-\log(\epsilon)\}$ and every half-ray with direction $\iota \in (0,\pi)$ tends to the puncture when $y\to +\infty$. Moreover, in the $w$-coordinate, the quadratic differential $q$ is written as
\[
	\zeta^{*}q=-R(1+O(e^{-y}))dw^{2} \ \ \ \text{for $y \to +\infty$} \ .
\]
As we saw above, the frame field $F:\tilde{\Sigma}\rightarrow \SL(4,\C)$ satisfies the system of ODE
\begin{align*}
	F^{-1}\frac{\partial F}{\partial x}&=\begin{pmatrix} -i\phi_{y} & 0 & e^{-\phi}\bar{q} & e^{\phi} \\
							0 & i\phi_{y} & qe^{-\phi} & e^{\phi} \\
							qe^{-\phi} & e^{-\phi}\bar{q} & 0 & 0 \\
							e^{\phi} & e^{\phi} & 0 & 0 
				             \end{pmatrix} \\
	F^{-1}\frac{\partial F}{\partial y}&=\begin{pmatrix} i\phi_{x} & 0 & -ie^{-\phi}\bar{q} & ie^{\phi} \\
							0 & -i\phi_{x} & iqe^{-\phi} & -ie^{\phi} \\
							iqe^{-\phi} & -ie^{-\phi}\bar{q} & 0 & 0 \\
							-ie^{\phi} & ie^{\phi} & 0 & 0 
				             \end{pmatrix} \ .
\end{align*}
Using the asymptotics of the conformal factor $\phi$ provided by Proposition \ref{prop:existence_sol}, for $y \to +\infty$ we can write
\begin{align*}
	F^{-1}\frac{\partial F}{\partial x}=A+O(e^{-2\alpha y}) \\
	F^{-1}\frac{\partial F}{\partial y}=B+O(e^{-2\alpha y})
\end{align*}
where
\begin{align*}
	A&=\begin{pmatrix} 0 & 0 & -\bar{R}|R|^{-\frac{1}{2}} & |R|^{\frac{1}{2}} \\
		           0 & 0 & -R|R|^{-\frac{1}{2}} & |R|^{\frac{1}{2}} \\
			   -R|R|^{-\frac{1}{2}} & -\bar{R}|R|^{-\frac{1}{2}} & 0 & 0 \\
			   |R|^{\frac{1}{2}} & |R|^{\frac{1}{2}} & 0 & 0 \\
	\end{pmatrix} \\
	B&=\begin{pmatrix} 0 & 0 & i\bar{R}|R|^{-\frac{1}{2}} & i|R|^{\frac{1}{2}} \\
			   0 & 0 & -iR|R|^{-\frac{1}{2}} & -i|R|^{\frac{1}{2}} \\
			  -iR|R|^{-\frac{1}{2}} & -i\bar{R}|R|^{-\frac{1}{2}} & 0 & 0 \\
			  -i|R|^{\frac{1}{2}} & i|R|^{\frac{1}{2}} & 0 & 0 \\
	\end{pmatrix} \ .	
\end{align*} 
We want to study the asymptotics of the solution of the above system along every half-ray with direction $\iota \in (0, \pi)$ as $y \to +\infty$. First, we introduce a new change of coordinates, in order to relate the above system to the frame field of the horospherical surface. Define $\xi^{2}=-R$ and consider $\omega=\xi w=\eta+i\tau$. In these new cooordinates we have
\[
	\frac{\partial}{\partial x}=\Re(\xi)\frac{\partial}{\partial \eta}+\Imm(\xi)\frac{\partial}{\partial \tau} 
\]
and the vector tangent to a half-ray in direction $\iota \in (0,\pi)$ is given by
\[
	\frac{\partial}{\partial r}=\cos\iota\frac{\partial}{\partial x}+\sin\iota\frac{\partial}{\partial y}=\Re(\xi e^{i\iota})\frac{\partial}{\partial \eta}+\Imm(\xi e^{i\iota})\frac{\partial}{\partial \tau} \ .
\]
The system of ODE for the frame field $\tilde{F}=(\tilde{\sigma}_{\omega}, \tilde{\sigma}_{\bar{\omega}}, N, \tilde{\sigma})$ becomes 
\begin{align}\label{eq:newframePDE}
	\tilde{F}^{-1}\frac{\partial \tilde{F}}{\partial x}&=SDS^{-1}+O(e^{-2r\alpha\sin\iota})\\
	\tilde{F}^{-1}\frac{\partial \tilde{F}}{\partial r}&=SD'S^{-1}+O(e^{-2r\alpha\sin\iota})  \ ,
\end{align}
where 
\[
	D=\diag(\rho_{1}, \rho_{2}, \rho_{3}, \rho_{4}) \ \ \  \text{and} \ \ \ D'=\diag(\mu_{1}, \mu_{2}, \mu_{3}, \mu_{4}) \ ,
\]
with
\begin{align*}
	\rho_{1}&=-\rho_{3}=2\Re(\xi)          \ \ \ \ \ \ \ \ \ \ \ \ \mu_{1}=-\mu_{3}=2\Re(\xi e^{i\iota})          \\
	\rho_{2}&=-\rho_{4}=2\Imm(\xi)          \ \ \ \ \ \ \ \ \ \ \ \  \mu_{2}=-\mu_{4}=2\Imm(\xi e^{i\iota})              \ .
\end{align*}
The theory of ODE with parameters (\cite{ODE}, see also \cite[Appendix A]{Loftin_compactify}) implies that, given an initial condition $B_{0}\in \SO_{0}(2,2)$ that expresses the difference between the frame field of the maximal immersion $\tilde{\sigma}$ and that of the horospherical surface at the base point $\omega_{0}$, the solution to Equation (\ref{eq:newframePDE}) can be written for $r \to +\infty$ as
\begin{equation}\label{eq:framefield_asy}
	\tilde{F}(x,r)=B_{0}A_{0}S(\diag(e^{\mu_{1}r}, e^{\mu_{2}r}, e^{\mu_{3}r}, e^{\mu_{4}r})+o(\diag(e^{\mu_{1}r}, e^{\mu_{2}r}, e^{\mu_{3}r}, e^{\mu_{4}r}))S^{-1} \ .
\end{equation}
Therefore, recalling that the maximal embedding $\tilde{\sigma}$ can be recovered from the last column of $\tilde{F}$, we can conclude that
\begin{equation}\label{eq:embedding}
	\tilde{\sigma}(x,r)=B_{0}\begin{pmatrix}
				\sinh(\mu_{1}r)\\
				 \sinh(\mu_{2}r) \\
				 \cosh(\mu_{3}r)\\
				 \cosh(\mu_{4}r)
				\end{pmatrix}
			+o\left(\begin{pmatrix}
				\sinh(\mu_{1}r)\\
				 \sinh(\mu_{2}r) \\
				 \cosh(\mu_{3}r)\\
				 \cosh(\mu_{4}r)
				\end{pmatrix}\right) \ \  \ \text{as $r \to +\infty$} \ .
\end{equation}
This already implies that $\tilde{\sigma}$ approaches a saw-tooth in $\partial_{\infty}\AdS_{3}$ according to Table \ref{table:1} with direction $\theta=\iota+\arg(\xi)$. We remark that, since $\iota \in (0,\pi)$, only three points of the light-like polygon $B_{0}(\Delta)$ and only one saw-tooth actually appear in the boundary at infinity of $\tilde{\sigma}(\tilde{\Sigma})$. We want to relate these limit points with the fixed points of the holonomy $\rho(\gamma)$. By the discussion in Section \ref{subsec:holonomy_alt}, we know that the holonomy along the peripheral curve $\gamma$ is related to the limit
\[
	\lim_{r \to +\infty} M_{r}H_{r}M_{r}^{-1} \ ,
\]
where $M_{r}$ is the holonomy along the path from $w_{0}=(0, y_{0})$ to $(r\cos(\iota), r\sin(\iota))$, and the exact element in the conjugacy class is determined by the initial conditions. Again from the theory of asymptotics of ODE with parameters, Equation (\ref{eq:newframePDE}) implies that
\[
	M_{r}=S(\diag(e^{\mu_{1}r}, e^{\mu_{2}r}, e^{\mu_{3}r}, e^{\mu_{4}r})+o(\diag(e^{\mu_{1}r}, e^{\mu_{2}r}, e^{\mu_{3}r}, e^{\mu_{4}r})))S^{-1}
\]
as $r \to +\infty$. On the other hand, we already remarked that $\lim_{r \to +\infty}H_{r}=e^{2\pi A}$, thus
\[
	M_{r}H_{r}M_{r}^{-1}=S\begin{pmatrix}
				e^{2\pi\rho_{1}}+o(1) & o(e^{(\mu_{1}-\mu_{2})r}) & o(e^{(\mu_{1}-\mu_{3})r}) & o(e^{(\mu_{1}-\mu_{4})r}) \\
				o(e^{(\mu_{2}-\mu_{1})r}) & e^{2\pi\rho_{2}}+o(1) & o(e^{(\mu_{2}-\mu_{3})r}) & o(e^{(\mu_{2}-\mu_{4})r}) \\
				o(e^{(\mu_{3}-\mu_{1})r}) & o(e^{(\mu_{3}-\mu_{2})r}) & e^{2\pi\rho_{3}}+o(1) & o(e^{(\mu_{3}-\mu_{4})r}) \\
				o(e^{(\mu_{4}-\mu_{1})r}) & o(e^{(\mu_{4}-\mu_{2})r}) & o(e^{(\mu_{4}-\mu_{3})r}) & e^{2\pi\rho_{4}}+o(1)
				\end{pmatrix}S^{-1}
\] 
and the action of the holonomy $\rho(\gamma)$ at every point $p \in \AdS_{3}\cup \partial_{\infty}\AdS_{3}$ can be computed as
\begin{equation}\label{eq:holonomy}
	\rho(\gamma)p=\lim_{r \to +\infty}B_{0}A_{0}M_{r}H_{r}M_{r}^{-1}A_{0}^{-1}B_{0}^{-1}p \ ,
\end{equation}
because the frame field of $\tilde{\sigma}$ at the base point is $B_{0}A_{0}$ by assumption. This is sufficient to conclude that the vertices of the saw-tooth found above are fixed points of the holonomy in $\partial_{\infty}\AdS_{3}$. Infact, all the paths with direction $\iota$ are homotopic relative to the base point $\omega_{0}$, hence the holonomy, and in particular the eigenvectors and the eigenvalues, does not change. On the other hand, by varying the direction $\iota \in (0,\pi)$ we discover three different eigenvalues and the corresponding eigenvectors. (The last ones are completely determined by the $\SO_{0}(2,2)$-simmetry.) Therefore, the holonomy along the peripheral curve is given by
\[
	\rho(\gamma)=B_{0}A_{0}S\diag(e^{2\pi\rho_{1}}, e^{2\pi\rho_{2}}, e^{2\pi\rho_{3}}, e^{2\pi\rho_{4}})S^{-1}A_{0}^{-1}B_{0}^{-1}
\]
Moreover, we notice that two eigenvectors always correspond to the attractive and repulsive fixed point of the holonomy and the remaining vertex can be the eigenvector for either the second biggest or the second smallest eigenvalue of $\rho(\gamma)$ (see Table \ref{table:2}). In particular, if we keep the imaginary part of the residue fixed and we change the sign of the real part a future-directed saw-tooth becomes a past-directed saw-tooth.

\medskip
\begin{table}[!htb]
\begin{center}
\begin{tabular}{l l l l l}
\hline
\textbf{Sign of $\Re(R)$} & \textbf{Sign of $\Imm(R)$} & \textbf{Eigenvalue} &\textbf{Saw-tooth}\\
\hline Positive & Positive & Second biggest & Future-directed\\
	Positive & Negative & Second smallest & Future-directed\\
	Negative & Positive & Second smallest & Past-directed\\
	Negative & Negative & Second biggest & Past-directed\\
\hline\\
\end{tabular}
\end{center}
\caption{\small{The second vertex of the limiting triangle depends on the sign of $\Re(R)\Imm(R)$. The time-orientation of the saw-tooth is determined by the sign of $\Re(R)$}}\label{table:2}
\end{table}

\section{Parameterising regular anti-de Sitter structures}\label{sec:AdS_regularGHM}
From the results of the previous sections we can construct a globally hyperbolic anti-de Sitter structure from the data of a complete hyperbolic metric $h$ of finite area on $\Sigma$ and a regular meromorphic quadratic differential $q$. Namely, Theorem \ref{cor:riassunto} provides a unique equivariant maximal embedding into $\AdS_{3}$ whose boundary at infinity is a locally achronal curve $\Gamma(h,q)$ that contains the limit set of the holonomy and is completed to a topological circle by inserting light-like segments according to the residue at the corresponding puncture. Let $\Omega(h,q)$ be the domain of dependence of this boundary curve. The holonomy representation acts properly discontinuously on $\Omega(h,q)$ (\cite{MR2443264}, \cite{MR2369412}) and the quotient is the desired GHM anti-de Sitter manifold $M(h,q)$ diffeomorphic to $\Sigma \times \R$. On the other hand, for a fixed admissible representation $\rho: \pi_{1}(\Sigma) \rightarrow \dPSL$ (i.e. the projections $\rho_{l,r}$ are both faithful and discrete), the space of GHM anti-de Sitter structures $\mathcal{GH}(\Sigma)$ on $\Sigma \times \R$ is quite large: if $\Lambda_{\rho}$ is the limit set of the action of $\rho$, then there is a one-to-one correspendence between elements of $\mathcal{GH}(\Sigma)$ and $\rho(\pi_{1}(\Sigma))$-equivariant completions of $\Lambda_{\rho}$ to an achronal topological circle (\cite{bsk_multiblack}). The aim of this section is thus to characterise the image of the map
\begin{align*}
	\Psi: \mathcal{RMQ}(\Sigma) &\rightarrow \mathcal{GH}(\Sigma)\\
		(h,q) &\mapsto M(h,q)
\end{align*}
associating to an element $(h,q)\in \mathcal{RMQ}(\Sigma)$ of the bundle of regular meromorphic quadratic differentials over $\Teich(\Sigma)$ the corresponding GHM anti-de Sitter structure. 

\begin{prop}\label{prop:injectivity}The map $\Psi$ is injective.
\end{prop}
\begin{proof}Suppose by contradiction that $\Psi$ is not injective. Then we can find $(h,q)\neq (h',q') \in \mathcal{RMQ}(\Sigma)$ such that $\Psi(h,q)=\Psi(h',q')$. By definition, this means that the equivariant maximal embeddings associated to $(h,q)$ and $(h',q')$ have the same holonomy representation and the same boundary at infinity. On the other hand, the same argument as in \cite[Lemma 4.2]{Tambu_poly} shows that given a locally achronal curve $\Gamma$ in $\partial_{\infty}\AdS_{3}$ the maximal surface bounding $\Gamma$ is unique. This gives a contradiction because the couple $(h,q)$ in uniquely determined by the embedding data of the maximal surface.
\end{proof}

\begin{prop}\label{prop:continuity}The map $\Psi$ is continuous.
\end{prop}
\begin{proof}Let us first specify the topology that we consider on $\mathcal{GH}(\Sigma)$. As explained above, $\mathcal{GH}(\Sigma)$ is in one-to-one correspondende with the set of couples $(\rho, \Gamma_{\rho})$, where $\rho:\pi_{1}(\Sigma) \rightarrow \dPSL$ is an admissible representation and $\Gamma_{\rho}$ is a $\rho(\pi_{1}(\Sigma))$-equivariant completion of the limit set of $\rho$ to a locally achronal topological circle in $\partial_{\infty}\AdS_{3}$. We thus consider on $\mathcal{GH}(\Sigma)$ the topology induced by the product of the usual topology in the space of representations and the Hausdorff topology for compact sets in $\partial_{\infty}\AdS_{3}$. \\
\indent Now, let $(h_{n},q_{n})\in \mathcal{RMQ}(\Sigma)$ be a sequence converging to $(h,q) \in \mathcal{RMQ}(\Sigma)$. We need to prove that the holonomy representation of $M(h_{n},q_{n})$ converges to the holonomy representation of $M(h,q)$ and the boundary curve $\Gamma(h_{n},q_{n})$ converges to $\Gamma(h,q)$ in the Hausdorff topology. Let $v_{n}$ and $v$ be the solution to Equation (\ref{eq:PDE}) associated to the data $(h_{n},q_{n})$ and $(h,q)$, respectively. On every compact set $K\subset \Sigma$, the supersolution and the subsolution found in Proposition \ref{prop:existence_sol} provide a uniform bound for $\Delta_{h_{h}}v_{n}$. Since $h_{n}$ is a convergent sequence, standard theory for elliptic PDE gives a uniform $W^{1,2}$ bound for $v_{n}$. Thus $v_{n}$ subconverges to a weak solution of the equation
\[
	\frac{1}{2}\Delta_{h}v=e^{2v}-e^{-2v}\|q\|^{2}_{h}+\frac{1}{2}K_{h} \ ,
\]
in $W^{1,2}$ on every compact set. By elliptic regularity $v$ is smooth and the convergence is actually smooth. We deduce that the embedding data of the unique maximal surface in $M(h_{n},q_{n})$ converges smoothly on compact sets to the embedding data of the unique maximal surface in $M(h,q)$. By lifting to the universal cover, this implies that the corresponding equivariant maximal embeddings $\tilde{\sigma}_{n}:\tilde{\Sigma}\rightarrow \AdS_{3}$ are converging smoothly on compact sets (up to post-composition by a global isometry) to $\tilde{\sigma}:\tilde{\Sigma}\rightarrow \AdS_{3}$, and thus the boundary at infinity $\Gamma(h_{n},q_{n})$ converges to $\Gamma(h,q)$ in the Hausdorff topology. The convergence of the holonomy representation follows from the general result below.
\end{proof}

\begin{lemma}\label{lm:conv_hol} Let $\tilde{\sigma}_{n}:\tilde{\Sigma}\rightarrow \AdS_{3}$ be a sequence of $\rho_{n}$-equivariant space-like embeddings. If $\tilde{\sigma}_{n}$ converges to a space-like embedding $\tilde{\sigma}$ smoothly on compact sets, then $\rho_{n}$ converges, up to subsequences, to a representation $\rho$ and $\tilde{\sigma}$ is $\rho$-equivariant.
\end{lemma}
\begin{proof}Fix a base point $p \in \tilde{\Sigma}$. Let $\{\gamma_{j}\}_{j=1}^{m}$ be a finite generating set of $\pi_{1}(\Sigma,p)$. Let $K\subset \tilde{\Sigma}$ be a compact set such that $\{\gamma_{j}\cdot p\}_{j=1}^{m} \subset K$. The holonomy representations $\rho_{n}:\pi_{1}(\Sigma)\rightarrow \Isom(\AdS_{3})$ are completely determined by the frame field $\tilde{F}_{n}$ of $\tilde{\sigma}_{n}$ at the base point $p$ and the  collection of matrices $\{M_{n,\gamma_{i}}\}_{i=1}^{m}$ sending $\tilde{F}_{n}(\tilde{\sigma}_{n}(p))$ to $\tilde{F}_{n}(\tilde{\sigma}_{n}(\gamma_{i}\cdot p))$. Since $\tilde{\sigma}_{n}$ converges smoothly on compact sets, we have that $\tilde{F}_{n}(\tilde{\sigma}_{n}(p))$ converges to $\tilde{F}(\tilde{\sigma}(p))$ and $\tilde{F}_{n}(\tilde{\sigma}_{n}(\gamma_{i}\cdot p))$ converges to $\tilde{F}(\tilde{\sigma}(\gamma_{i}\cdot p))$ for every $i=1, \dots, m$, where $\tilde{F}$ is the frame field of the embedding $\tilde{\sigma}$. Since $\tilde{\sigma}$ is space-like, the sequences $M_{n,\gamma_{i}}$ are contained in a compact set of $\SO_{0}(2,2)$ for every $i$, hence they converge to some $M_{\gamma_{i}}$ up to subsequences. Together with the previous remark, this shows that $\rho_{n}$ subconverges to a representation $\rho$. Taking then the limit of the expression
\[
	\tilde{\sigma}_{n}(\gamma \cdot p)=\rho_{n}(\gamma)\tilde{\sigma}_{n}(p) \ ,
\]
we conclude that $\tilde{\sigma}$ is $\rho$-equivariant.
\end{proof}

We define the subset of \emph{regular} GHM anti-de Sitter structures on $\Sigma\times \R$ as the image of the map $\Psi$:
\[
	\mathcal{GH}^{reg}(\Sigma)=\Psi(\mathcal{RMQ}(\Sigma))\subset \mathcal{GH}(\Sigma) \ .
\]
From Section \ref{sec:boundary} we know that the curves $\Gamma(h,q)$ are always obtained by completing the limit set of the holonomy representation with light-like segments in a precise way. However, we do not know if any admissible representation $\rho:\pi_{1}(\Sigma) \rightarrow \dPSL$ is attained in the image. To this aim, we construct a bijection between $\mathcal{GH}^{reg}(\Sigma)$ and another parameter space. Let $\mathcal{F}(\Sigma)$ denote the Fricke space of $\Sigma$. We define $\widehat{\mathcal{DF}(\Sigma)}$ as the set of $(k+2)$-uples (recall that $k$ is the number of punctures of $\Sigma$) of the form $(h_{l},h_{r}, \epsilon_{1}, \dots, \epsilon_{k})$, where $h_{l},h_{r} \in \mathcal{F}(\Sigma)$ and $\epsilon_{j}$ is a decoration on each puncture so that
\[
	\epsilon_{j}=\begin{cases}
		\pm 1 \ \ \ \text{if the puncture $p_{j}$ is a geodesic boundary for both $h_{l}$ and $h_{r}$} \\
	0 \ \ \ \ \ \text{otherwise}
		\end{cases} \ .
\]
Notice that each $(k+2)$-uple uniquely determines a GHM anti-de Sitter structure on $\Sigma \times \R$ in the following way. Let $\rho_{l,r}:\pi_{1}(\Sigma)\rightarrow \PSL(2,\R)$ be the holonomy representation of $h_{l,r}$. By definition, the representation $\rho=(\rho_{l},\rho_{r})$ is admissible. Let $\Lambda_{\rho}$ be its limit set in $\partial_{\infty}\AdS_{3}$. The decoration $\epsilon_{j}$ gives a unique way to complete $\Lambda_{\rho}$ to a topological circle: for every puncture $p_{j}$ corresponding to a geodesic boundary, we add a future-directed (resp. past-directed) saw-tooth if $\epsilon_{j}=1$ (resp. $\epsilon_{j}=-1$) and then connect (if necessary) all remaining causally related points with light-like segments. The quotient of the domain of dependence of the resulting curve by the action of $\rho$ gives a regular GHM anti-de Sitter structure on $\Sigma\times \R$. We endow $\widehat{\mathcal{DF}(\Sigma)}$ with the topology induced by this bijection.

\begin{teo}\label{thm:hol_para}There is a bijection between $\mathcal{GH}^{reg}(\Sigma)$ and $\widehat{\mathcal{DF}(\Sigma)}$.
\end{teo}
\begin{proof}Fix a decoration $(\epsilon_{1}, \dots, \epsilon_{k})$ and suppose that $0\leq m\leq k$ elements are non-zero and the other $k-m$ vanish. For every choice of a subset $S_{j}$ of $0\leq j\leq k-m$ punctures with vanishing decoration, the set of couples of hyperbolic metrics compatible with the decoration and the choice of $S_{j}$ is parameterised by 
\[
	\left(\Teich(\overline{\Sigma}_{m+j,k-m-j})\times \Teich(\overline{\Sigma}_{m,k-m})\right) \bigcup  
	\left(\Teich(\overline{\Sigma}_{m,k-m})\times \Teich(\overline{\Sigma}_{m+j,k-m-j})\right)
\]
where $\Teich(\overline{\Sigma}_{a,b})$ denotes the Teich\"muller space of hyperbolic metrics with $a$ geodesic boundary components and $b$ cusps. We construct a bijection between each piece of the above union and a subbundle of $\mathcal{RMQ}(\Sigma)$. We will explain the details for $\Teich(\overline{\Sigma}_{m+j,k-m-j})\times \Teich(\overline{\Sigma}_{m,k-m})$ and we then indicate what needs to be changed for the other case. We consider the subbundle $X_{j}$ of $\mathcal{RMQ}(\Sigma)$ whose fibre $F_{j}$ over $h \in \Teich(\Sigma)$ consists of regular meromorphic quadratic differentials satisfying the following rules:
\begin{enumerate}[i)]
	\item if $\epsilon_{i}=+1$, then the real part of the residue $R_{i}$ at the puncture $p_{i}$ is positive;
	\item if $\epsilon_{i}=-1$, then the real part of the residue $R_{i}$ at the puncture $p_{i}$ is negative;
	\item if $\epsilon_{i}=0$ and $p_{i} \in S_{j}$, then $\Re(R_{i})=0$ and $\Ima(R_{i})>0$;
	\item if $\epsilon_{i}=0$ and $p_{i} \notin S_{j}$, then $R_{i}=0$.
\end{enumerate}
Notice that both $X_{j}$ and $\Teich(\overline{\Sigma}_{m+j,k-m-j})\times \Teich(\overline{\Sigma}_{m,k-m})$ are manifolds of the same dimension:
\begin{align*}
	\dim X_{j}&=\dim(\Teich(\Sigma))+\dim(F_{j})\\
		  &=(3|\chi(\overline{\Sigma})|+2k)+(3|\chi(\overline{\Sigma})|+2k+2m+j)\\
		  &=6|\chi(\overline{\Sigma})|+4k+2m+j \\
	          &=\dim(\Teich(\overline{\Sigma}_{m+j,k-m-j})\times \Teich(\overline{\Sigma}_{m,k-m})) \ .
\end{align*}		      
From Proposition \ref{prop:continuity} and Proposition \ref{prop:injectivity}, if suffices to prove that the natural map
\begin{align*}
	\hol \circ \Psi:X_{j} &\rightarrow \Teich(\overline{\Sigma}_{m+j,k-m-j})\times \Teich(\overline{\Sigma}_{m,k-m})\\
		(h,q) &\mapsto \hol(\Psi(h,q))
\end{align*}
is proper. Let $(h_{n},q_{n}) \in X_{j}$ be a sequence such that $\hol(\Psi(h_{n},q_{n}))$ converges. Since the way of completing the limit set of the holonomy is determined by the decoration, this implies that the boundary at infinity of the maximal surfaces with embedding data $I_{n}=2e^{2v_{n}}h_{n}$ and $II_{n}=2\Re(q_{n})$ are converging in the Hausdorff topology. The techniques introduced in \cite[Section 4.1]{Tambu_poly} show that the maximal surfaces bounding such curves are actually converging smoothly on compact sets. Hence their embedding data converge and $\Psi$ is proper. \\
\indent For the other case, it is sufficient to require $\Ima(R_{i})<0$ in $(iii)$ by Proposition \ref{prop:holonomy}. Since $\widehat{\mathcal{DF}(\Sigma)}$ is the disjoint union over all possible decorations and choices of the subset $S_{j}$ of all these submanifolds, $\Psi$ is surjective.
\end{proof}
  
\section{Application to minimal Lagrangian maps}\label{sec:application}
Let $\Omega_{r}, \Omega_{l} \subset \h^{2}$ be open domains of the hyperbolic plane. An orientation preserving diffeomorphism $m:\Omega_{l}\rightarrow \Omega_{r}$ is minimal Lagrangian if its graph is a minimal surface in $\h^{2}\times \h^{2}$ that is Lagrangian for the symplectic form $\omega_{\h^{2}}\oplus -\omega_{\h^{2}}$. \\

Minimal Lagrangian maps have been extensively studied when $\Omega_{r}=\Omega_{l}=\h^{2}$. For instance, if we ask $m$ to be equivariant under the action of two Fuchsian representations $\rho_{l},\rho_{r}:\pi_{1}(\overline{\Sigma}) \rightarrow \PSL(2,\R)$, a result by Schoen (\cite{Schoenharmonic}) states that such $m$ always exists and is unique in each isotopy class. Later, Bonsante and Schlenker (\cite{bon_schl}) used anti-de Sitter geometry to construct minimal Lagrangian maps from $\h^{2}$ to $\h^{2}$ with given boundary conditions. More precisely, they proved that every quasi-symmetric homeomorphism of the circle is realised on the boundary of a unique minimal Lagrangian diffeomorphism of the hyperbolic plane. Here we use the techniques introduced by Bonsante and Schlenker in order to construct a class of minimal Lagrangian maps between hyperbolic surfaces with cusps and geodesic boundary. \\

Let $\Sigma$ still denote a surface with $k$ punctures and negative Euler characteristic. Let $h$ and $h'$ be hyperbolic structures on $\Sigma$ so that each puncture corresponds to a cusp or a geodesic boundary, and denote with $\rho$ and $\rho'$ the corresponding holonomy representations. Let $n$ be the common number of geodesic boundary components. By Theorem \ref{thm:hol_para}, we can find $2^{n}$ regular GHM anti-de Sitter manifolds with holonomy $(\rho, \rho')$. We are going to show that the maximal surface embedded into each of these manifolds corresponds to a minimal Lagrangian diffeomorphism from $(\Sigma, h)$ to $(\Sigma, h')$ with a precise behaviour on the boundaries, thus proving Theorem E. \\

Let us first recall the relation between equivariant maximal surfaces in anti-de Sitter space and minimal Lagrangian maps between hyperbolic surfaces. Let $\tilde{S}$ be a $(\rho,\rho')$-equivariant maximal surface in $\AdS_{3}$ with second fundamental form $II=2\Re(q)$. The Gauss map 
\[
	G: \tilde{S} \rightarrow \h^{2}\times \h^{2}
\]
is harmonic for the conformal structure of the induced metric on $\tilde{S}$ and $(\rho, \rho')$-equivariant. Hence the two projections $G_{l}=\pi_{l}\circ G$ and $G_{r}=\pi_{r}\circ G$ are also harmonic. The bound on the principal curvature given in Theorem \ref{cor:riassunto} guarantees that these maps are local diffeomorphisms (\cite{bon_schl}). They are also injective due to the following:

\begin{lemma}Let $S$ be a maximal surface in $\AdS_{3}$ with principal curvatures in $(-1,1)$. Then the left and right Gauss maps $G_{l,r}$ are injective on $S$.
\end{lemma}
\begin{proof} Let $S_{r}$ be the surface obtained by pushing $S$ along the normal direction for a time $r\in \R$. The shape operator of the surface $S_{r}$ is given by
\[
	B_{r}=(\cos(r)E+\sin(r)B)^{-1}(-\sin(r)E+\cos(r)B) \ ,
\] 
hence the surface $S_{r}$ is smooth for every $r \in [-\pi/4, +\pi/4]$. Moreover, $S_{-\pi/4}$ is future-convex with constant curvature $-2$. Since $S_{-\pi/4}$ is equidistant to $S$, they have the same Gauss map. It is thus sufficient to prove that if $S'$ is a future-convex space-like surface in $\AdS_{3}$, then $G_{r}$ and $G_{l}$ are injective. Let $p, p' \in S'$. By assumption, the totally geodesic planes $T_{p}S'$ and $T_{p'}S'$ tangent to $S'$ at these points are space-like and $S'$ is contained in the intersection $\Omega$ of the future half-spaces bounded by $T_{p}S'$ and $T_{p'}S'$. The boundary $\partial \Omega$ is either a totally geodesic plane or a pleated surface with pleating locus made by a single geodesic. Since $\partial \Omega$ is tangent to $S'$ at $p$ and $p'$, the Gauss map of $S'$ coincides with the Gauss map of $\partial \Omega$ at those points. It is easy to verify that $G_{r}$ and $G_{l}$ are injective on $\partial \Omega$, thus we can conclude that $G_{r}(p)\neq G_{r}(p')$ and $G_{l}(p)\neq G_{l}(p')$.
\end{proof}

We deduce that $G_{r}$ and $G_{l}$ are diffeomorphisms onto their image, and the pull-back metrics $G_{l}^{*}g_{\h^{2}}$ and $G_{r}^{*}g_{\h^{2}}$ coincide with the lifts of $h$ and $h'$, respectively. Moreover, a direct computation shows that $G_{l}$ and $G_{r}$ have opposite Hopf differentials $\pm 2iq$ (see for instance \cite[Prop. 6.3]{Tambu_Qiyu}). Therefore, the composition
\[
	\tilde{m}=G_{r} \circ G_{l}^{-1}
\]
induces a minimal Lagrangian map $m:(\Sigma, h) \rightarrow (\Sigma, h')$. In particular, the harmonic maps into which $m$ factors are the harmonic diffeomorphisms from a Riemann surface with punctures to a hyperbolic surface with geodesic boundary or cusps, whose Hopf differential is meromorphic with poles of order at most $2$ at the punctures, as studied in \cite{Mike_infenergy}. \\

We want now to describe the behaviour of these minimal Lagrangian maps in a collar neighbourhood of a geodesic boundary of $(\Sigma, h)$. To this aim, it is sufficient to study the harmonic maps $G_{l}$ and $G_{r}$ in a neighbourhood of the corresponding puncture. Passing to the universal cover, this means that we need to determine the behaviour of the left and right Gauss maps along sequences that converge to a point on the boundary at infinity of the equivariant maximal surface $\tilde{S}$ lying on a light-like segment. From Equation (\ref{eq:framefield_asy}), $\tilde{S}$ is asymptotic to an isometric copy of the model horospherical surface in a neighbourhood of the puncture, thus $G_{l}$ and $G_{r}$ can be approximated by the Gauss map of the horospherical surface, which has been studied in \cite[Section 5]{Tambu_poly}. In order to recall that result, let us first introduce some notation. Identify $\h^{2}$ with a totally geodesic space-like plane $P_{0}$ in $\AdS_{3}$. Following the left and right ruling of $\partial_{\infty}\AdS_{3}$, we can define two projections $\pi_{r,l}:\partial_{\infty}\tilde{S} \rightarrow \partial_{\infty}P_{0}$ by sending $\xi \in \partial_{\infty}\tilde{S}$ to the unique intersection $\pi_{r,l}(\xi)$ between the line belonging to the left or right foliation passing through $\xi$ and $\partial_{\infty}P_{0}$. In particular, a light-like segment $e_{l}$ belonging to the left foliation is mapped to a point by $\pi_{l}$ and to a segment by $\pi_{r}$. From the computations of \cite{Tambu_poly} we deduce that any sequence of points approaching a light-like segment $e_{l}$ belonging to the left-foliation gets sent by the left Gauss map to a sequence of points in $\h^{2}$ limiting to $\pi_{l}(e_{l})$, and by the right Gauss map to sequences of points approaching the geodesic in $\h^{2}$ with points at infinity $\partial\pi_{r}(e_{l})$. In particular, these minimal Lagrangian maps cannot be extended to the geodesic boundaries.

\bibliographystyle{alpha}
\bibliographystyle{ieeetr}
\bibliography{bs-bibliography}

\bigskip

%\noindent \footnotesize \textsc{DEPARTMENT OF MATHEMATICS, RICE UNIVERSITY}\\
%\emph{E-mail address:}  \verb|andrea_tamburelli@libero.it|

\end{document}